\documentclass[11pt]{amsart}

\usepackage{amsmath,amscd,amssymb,latexsym,color}
\usepackage{MnSymbol}
\usepackage{longtable}

\allowdisplaybreaks

\newcommand{\sm}{\left(\smallmatrix}
\newcommand{\esm}{\endsmallmatrix\right)}
\newcommand{\mat}{\begin{pmatrix}}
\newcommand{\emat}{\end{pmatrix}}
\renewcommand{\c}{\mathfrak{c}}
\renewcommand{\t}{\tau}

\renewcommand{\i}{\infty}
\newcommand{\G}{\Gamma}
\newcommand{\g}{\gamma}

\newcommand{\lt}{\left}
\newcommand{\rt}{\right}
\newcommand{\A}{\mathbb A}
\newcommand{\Q}{\mathbb Q}
\newcommand{\Z}{\mathbb Z}
\newcommand{\C}{\mathbb C}
\newcommand{\R}{\mathbb R}

\renewcommand{\H}{\mathbb H}

\newtheorem{thm}{Theorem}
\newtheorem{lem}[thm]{Lemma}
\newtheorem{cor}[thm]{Corollary}
\newtheorem{prop}[thm]{Proposition}
\newtheorem{exam}[thm]{Example}

\theoremstyle{definition}
\newtheorem{df}[thm]{Definition} 
\newtheorem{rmk}[thm]{Remark}
\numberwithin{equation}{section}
\numberwithin{thm}{section}
\newcommand{\dotcup}{\ensuremath{\mathaccent\cdot\cup}}

\begin{document}

\title[HeckeSystem]{Hecke System of Harmonic Maass Functions and Applications to modular curves of higher genera}

\author{Daeyeol Jeon, Soon-Yi Kang and Chang Heon Kim}
\address{Department of Mathematics Education, Kongju National University, Kongju, 32588 Korea}
\email{dyjeon@kongju.ac.kr}
\address{Department of Mathematics, Kangwon National University, Chuncheon, 24341 Korea} \email{sy2kang@kangwon.ac.kr}
\address{Department of Mathematics, Sungkyunkwan University,
 Suwon, 16419 Korea}
\email{chhkim@skku.edu}

\begin{abstract}  In Monstrous moonshine, genus $0$ property and the notion of replicability are strongly connected. With regards to recent developments of moonshine, we investigate a higher genus generalization of replicability for a general automorphic form. Specifically, we extend the definitions of replicates and a Hecke operator to harmonic Maass functions on modular curves of higher genera to obtain number theoretic generalizations of important results in Monstrous moonshine. Furthermore, we show the utility of the extended notions in yielding uniform proofs for numerous arithmetic properties of Fourier coefficients of modular functions of arbitrary level, which have been proved only for special cases of curves of genus zero or small prime levels.

\end{abstract}
\maketitle

\renewcommand{\thefootnote}%
             {}
 {

\section{Introduction}

Monstrous moonshine, a relation between the Monster group and the Hauptmoduls, was proved by virtue of string theory. This classical moonshine contributed to the discovery of new algebraic structures such as vertex-operator algebras and generalized Kac-Moody algebras.
A resurgence of interest in moonshine began a decade ago, when mock modular forms and Mathieu groups were found to be deeply connected by K3 string theory.  New forms of moonshine are thus speculated to have impact on number theory and mathematical physics at large. In number theory, in particular, they offer rich arithmetic properties of mock modular forms and half-integral weight modular forms; for instance a relation between half-integral weight modular forms and the Birch-Swinnerton-Dyer conjecture.
For a detailed account on moonshine, we refer the reader to a survey on classical moonshine \cite{Gannon} and a recent survey including umbral moonshine \cite{DGOreview}. See also \cite{DMO-p, DMO-o} for non-Monstrous moonshine. 

Moonshine conjecture asserts that the McKay-Thompson series for the Monster group are Hauptmoduls. The conjecture was made by Conway and Norton \cite{CN}  and proved later by Borcherds \cite{B}. In the proof of the conjecture, {\it{replication formulas}} played an important role. Conway and Norton used the notion of {\it{replicates}} of a Hauptmodul and established a replication formula; a sort of recursive relation between Faber polynomials of a Hauptmodul and its replicates. Then they defined the generalized Hecke operator whose action on a Hauptmodul gives the Faber polynomials of the Hauptmodul. Using the replication formulas and properties of the Hecke operator, they could derive a recursion formula satisfied by the Fourier coefficients of a Hauptmodul and its replicates. (See \cite{Koike} for the full statements and proofs.) Borcherds used vertex operator algebras to show that the same recursion formula holds for the coefficients of the corresponding McKay-Thompson series and its replicates. 

Considering the recent developments of moonshine, it is natural to ask for a higher genus generalization for a general automorphic form of the arguments above.  Inferring from the fact that Faber polynomials of a Hauptmodul form a canonical basis of the space of weakly holomorphic modular functions on the modular group of genus $0$, one may expect canonical basis elements of the space of weakly holomorphic modular functions with an arbitrary level would play the role of  Faber polynomials.   

Let $\t$ be a value in the complex upper half plane $\mathbb H$ and $q=e^{2\pi i \t}$. Also, let $N$ be a positive integer and $S$ be a subset of exact divisors of $N$.
Following Koike \cite{Koike}, we denote by $\G=N+S$ the subgroup of ${\rm PSL}_2(\R)$ generated by $\G_0(N)$ and the Atkin-Lehner involutions $W_{Q,N}$ for all $Q\in S$. For convenience, we write $\G_0(N)$ or $N$ for $N+\emptyset$ and $\G_0(N)+$ or $N+$ when $S$ is the set of all Atkin-Lehner involutions.
Niebur \cite{Niebur} showed that for $\G$ of any level $N$, there are certain analytic continuations of Niebur-Poincar\'e series which are harmonic Maass functions on $\G$. We denote them by $j_{\G,n}(\t)=j_{N+S,n}(\t)$ ($n\in\mathbb{N}$), whose Fourier expansions are given by (\cite[Theorem 1]{Niebur} and \cite[Proposition 3.1]{BKLOR}) 
\begin{equation}\label{jnnfour}j_{\G,n}(\t)=q^{-n}-\bar q^n + c_\G(n,0)+\sum_{\ell\geq 1}(c_\G(n,\ell)q^\ell + c_\G(n,-\ell)\bar q^\ell).\end{equation}
Let $J_{\G,n}(\t):=j_{\G,n}(\t)-c_\G(n,0)$ and define  $J_{\G,0}(\t):=1$. According to Bruinier and Funke's work on harmonic Maass forms, the image of the non-holomorphic part of $J_{\G,n}(\t)$ under the differential operator $2i \overline{\frac{\partial}{\partial\bar \t}}$ is a weight 2 cusp form on $\G$.  Thus when the modular curve $X(\G)$ corresponding to $\G$ has genus zero, each $J_{\G,n}(\t)$ does not have a non-holomorphic part, which implies that it is just a weakly holomorphic modular function.
In particular,
$$J_{1,1}(\t)=J(\t)=q^{-1}+196884q+21493760q^2+\cdots$$
is the Hauptmodul on $\mathrm{SL}_2(\mathbb Z)$. 
%
Moreover,  $J_{1,n}(\t)=P_n(J(\t))=q^{-n}+O(q)$, where $P_n$ is the Faber polynomial, the unique polynomial for which $P_n(J(\t))-q^{-n}$ has a $q$-expansion with only strictly positive powers of $q$.  The functions $J_{1,n}(\t)$ ($n\geq 0$) form a basis for the space of weakly holomorphic modular functions on the full modular group and also form a Hecke system as proved in \cite{CN, Koike} by using the replication formula.  In other words, for the normalized Hecke operator $T(n)$, they satisfy that
\begin{equation}\label{hj}J_{1,n}(\t)=J_{1,1}(\t)|T(n)(\t).
\end{equation}
Utilizing this, Asai, Kaneko and Ninomiya \cite{AKN} proved that the Hauptmodul $J$ gives a weight 2 meromorphic modular form representation for the generating function of $J_{1,n}(\t)$.
Namely, if $q_1=e^{2\pi i \t_1}$ and $q_2=e^{2\pi i \t_2}$ for $\t_1,\tau_2\in\mathbb H$, then it holds that
\begin{equation}\label{gj}\sum_{n=0}^\infty J_{1,n}(\t_1)q_2^n=-\frac{1}{2\pi i}\frac{J'(\t_2)}{J(\t_2)-J(\t_1)}.\end{equation}
This well known identity is equivalent to the famous denominator formula for the Monster Lie algebra (or {\it{Koike-Norton-Zagier infinite product identity}}):
\begin{equation}\label{df}
J(\t_1)-J(\t_2)=(q_1^{-1}-q_2^{-1})\prod_{m,n=1}^\infty(1-q_1^mq_2^n)^{c_1(1,mn)}.
\end{equation}
The results above for $\mathrm{SL}_2(\mathbb Z)$ have been generalized to those for congruence subgroups  of genus zero by many mathematicians. (See, for example, \cite{Ye} and \cite{Car} and references therein for results on and generalizations of (\ref{gj}) and (\ref{df}), respectively.)

The main purpose of this paper is to generalize \eqref{hj} and \eqref{gj} for $J_{\G,n}(\t)$ of arbitrary $\G$.

We begin with generalizing the replication formulas for Hauptmoduls of Norton\,\cite{Norton} and Koike\,\cite{Koike}, from which we easily obtain the generalization of \eqref{hj} for $J_{\G,n}$.
Beforehand, we introduce necessary notions and notations. For any prime $p$, we define the group $\G^{(p)}$ as follows: 
\begin{equation}\label{gns}
\G^{(p)}:=N^{(p)}+S^{(p)},
\end{equation}
where $N^{(p)}=N/(p,N)$ and $S^{(p)}$ is the set of all $Q\in S$ which divide $N^{(p)}$.
In addition, for a positive integer $m$ with the prime decomposition $m=p_1p_2\cdots p_r$, we define $\G^{(m)}$ by
\begin{equation}\label{gp}
\G^{(m)}:=\G^{(p_1)(p_2)\cdots(p_r)}.
\end{equation}
Now, we define the {\it $m$-plicate} of $J_{\G,n}$ by
\begin{equation}\label{defp}J_{\G,n}^{(m)}:=J_{\G^{(m)},n}.\end{equation}
Also, we define $J_{\G,\frac{n}{p}}=0$ if $p\nmid n$. The first main theorem is the {\it $p$-plication formula}, which is essential in our construction of a generalized Hecke operator that acts on harmonic Maass functions.
\begin{thm}[{\it{$p$-plication formula}}] \label{p-pli} For any prime $p$, we have that
\begin{equation}\label{ep-pli}
J_{\G,n}|{U^*_p}(\t)+J_{\G,n}^{(p)}(p\t)=J_{\G,pn}(\t)+pJ_{\G,\frac{n}{p}}^{(p)}(\t),
\end{equation}
where the $U^*_p$-operator that acts on a complex valued function $f$ on $\H$ is given by
\begin{equation}\label{ups}f|{U^*_p}(\t)=pf|{U_p}(\t)=\sum_{i=0}^{p-1} f\left( \frac{\t+i}{p} \right).\end{equation} 
\end{thm}

It is interesting to note that \eqref{ep-pli} takes the identical form as the $p$-plication formula for Faber polynomials of a Hauptmodul in Theorem 1.1 of \cite{Koike}.  We, therefore, adopt Koike's arguments from \cite{Koike} for defining a Hecke operator below.
\begin{df} For any prime number $p$ and positive integer $r$, we define the Hecke operator $T(p^r)$ acting on $J_{\G,n}$ by
\begin{equation}\label{deft}J_{\G,n}|T(p^r)(\t):=\sum_{i=0}^{r} J_{\G,n}^{(p^i)}|{U^*_{p^{r-i}}}(p^i\t).\end{equation}
\end{df}
If $p\nmid N$ and $r=1$, then this definition gives the same action with the classical Hecke operator $T_p$: 
$${pJ_{\G,n}|T_p(\t)=J_{\G,n}|{U^*_p}(\t)+J_{\G,n}|{V_p}(\t)},$$
where $V_p=\sm p&0\\0&1\esm$.  Here, for any $\gamma\in{\rm GL}_2(\R)$, we define $f|\gamma(\t)=f(\gamma\t)$, a linear fractional transformation.  After we extend the definition of the Hecke operator $T(m)$ to any positive integer $m$ in Section 4, we generalize the {\it $p$-plication formula} to the following {\it replication formula}:
\begin{thm}[{\it{replication formula}}] \label{replic} For any positive integer $m$, we have
\begin{align} J_{\G,n}|T(m)(\t)&=\sum_{d|m}J_{\G,n}^{(d)}|U^*_{\frac{m}{d}}(d\t) \label{Hecke3-2}\\
&=\sum_{d|(n,m)}dJ_{\G,\frac{mn}{d^2}}^{(d)}(\t) \label{Hecke3-1}. 
\end{align}
\end{thm}

From \eqref{Hecke3-1}, we find that $J_{\G,n}$'s form a Hecke system as well.

\begin{cor} \label{ThjN} For any positive integer $n$, we have
\begin{equation}\label{hjN}
J_{\G,n}(\t)=J_{\G,1}|T(n)(\t).
\end{equation}
\end{cor}

\begin{rmk} 
(1) Norton \cite{Norton} associated a modular function $f(\t)=q^{-1}+O(q)$ with another modular function with the same form, called the replicate of $f$. Our definition of replicates of $J_{\G,n}$ is consistent with this concept. Moreover, as the same result with \eqref{hjN} ensures that a genus $0$ modular function $J_{\G,1}$ is completely replicable (for example, see \cite{ACMS}), we may extend the notion of complete replicability to  $J_{\G,1}$ in general via \eqref{hjN}.

(2) Niebur-Poincar\'e seris is the Maass-Poincar\'e series of weight $0$. Proposition 3.1 of \cite{BM} gives the proof of the same result with Corollary \ref{ThjN} for Maass-Poincar\'e series on $\G_0(N)$ of any non-positive weights to establish dimension formulas for certain vertex operator algebras.  Our proof can also be extended to Maass-Poincar\'e series. Special cases of Corollary \ref{ThjN} are already known when  $(N,n)=1$.  For example, Lemma 5.1 of \cite{KK} gives its special case when $\G=\G_0(N)+$ and Theorem 1.1 (2) of \cite{BKLOR} does when $\G=\G_0(N)$.  

\end{rmk}

The notions of {\it{replicates}} and the Hecke operator above allow one to find arithmetic properties of Fourier coefficients of weakly holomorphic modular functions.  By \cite[Theorem 6]{Niebur}, any weakly holomorphic modular function which has a possible pole only at $i\i$ can be written as a sum of linear combinations of $J_{\G,n}$'s.
\begin{df} Let $p$ be a prime and $r\in \mathbb N$.  For any function of the form
$f=\sum_{n=0}^\ell a_n J_{\G,n},\,  (a_n\in\C),$
we define 
\begin{equation}\label{heckedef}
f|T(p^r):=\sum_{n=0}^\ell a_n J_{\G,n}|T(p^r).
\end{equation}
Also for any positive integer $m$, we define the {\it $m$-plicate} of $f$ by
\begin{equation}\label{def-fn}
f^{(m)}:=\sum_{n=0}^\ell a_n J_{\G,n}^{(m)}.
\end{equation}
\end{df}
Let $M_k^!(\G)$ denote the space of weakly holomorphic modular forms on $\G$ and $M_k^{!,\infty}(\G)$ ($S_k^{!,\infty}(\G)$, resp.) be its subspace consisting of the weakly holomorphic modular forms whose poles are supported only at $i\i$ and are holomorphic (vanish, resp.) at other cusps. We determine a couple of conditions for the Hecke operator to preserve holomorphicity. 
\begin{thm}\label{tpm} Let $\G=N+S$ and $p\nmid N$ be a prime.  
Then for any positive integer $r$, the Hecke operator $T(p^r)$ preserves $M_0^{!, \infty}(\G)$.
\end{thm}

\begin{thm}\label{heckepres}  
Let $n\in\mathbb N$ and $f\in M_0^{!,\infty}(\G)$.  If $f^{(d)}\in M_0^{!,\infty}(\G^{(d)})$ for all $d\mid n$, then $f|T(n)$ is weakly holomorphic.
\end{thm}

In classical moonshine, genus $0$ property and the notion of {\it{replicability}} are strongly connected.  We find out that a genus $0$ replicable function and a weakly holomorphic modular function that satisfies the hypothesis of the statement in Theorem \ref{heckepres} share a similar property.    

\begin{thm}\label{ffp} Let $p$ be a prime.  If $f\in M_0^{!,\infty}(\G)$ and $f^{(p)}\in M_0^{!,\infty}(\G^{(p)})$ both have integer coefficients, then $f\equiv f^{(p)} \pmod p$.
\end{thm}

In order to present our arguments in a more concrete form,  we construct a basis for $M_0^{!,\i}(\G)$.  For simplicity, assume that $i\i$ is not a Weierstrass point of $X(\G)$.  If $m\geq g+1$, when $g:=g(\G)$ is the genus of $X(\G)$ with $g>0$, there is always a weakly holomorphic modular function that has only pole at $i\i$ of order $m$ by Weierstrass gap theorem. By performing Gauss elimination on the coefficients of these functions, we can obtain the unique modular function of the form 
\begin{equation}\label{mfg}f_{\G,m}=q^{-m}+\sum_{l=1}^{g} a_{\G}(m,-l)q^{-l}+O(q)
\end{equation}
for each $m\geq g+1$. For $1\leq m\leq g$, we define $f_{\G,m}=0$ and $f_{\G,0}=1$. Then $\{f_{\G,m}:m\in\mathbb Z_{\geq0}\}$ forms a {\it reduced row echelon basis} of $M_0^{!,\infty}(\G)$. We note that $a_{\G}(m,-m)=-1$ for $1\leq m\leq g$ and also note that $f_{\G,\frac{m}{p}}=0$ unless $p$ divides $m$. When $g=0$, obviously we should define $f_{\G,m}:=J_{\G,m}$. Since $M^{!,\i}_0(\G)$ is generated by $J_{\G,n}(\t)$, it follows from \eqref{jnnfour} and \eqref{mfg} that 
\begin{equation}\label{fjg}f_{\G,m}=J_{\G,m}+\sum_{l=1}^{g} a_{\G}(m,-l)J_{\G,l}.\end{equation}

\begin{exam} Consider $$f_{22,3}(\t)=J_{22,3}(\t)+J_{22,1}(\t)=q^{-3}+q^{-1}+2q+2q^2+2q^5+2q^6+q^7+2q^9-\cdots\in M_0^{!,\infty}(22).$$
Its 2-plicate  $f_{22,3}^{(2)}=J_{11,3}+J_{11,1}$ equals $f_{11,3}$, which has the Fourier expansion
$$f_{11,3}(\t)=q^{-3}+q^{-1}+2q+2q^2+16q^3+16q^4+18q^5-46q^6-31q^7+48q^8-78q^9+\cdots\in M_0^{!,\infty}(11).$$
These two have the same principal parts and satisfy 
$$f_{22,3}(\t)\equiv f_{22,3}^{(2)}(\t)\pmod{2}$$
as expected by Theorem \ref{ffp}.
On the other hand, the function 
$$\left(\frac{\eta(\t)}{\eta(2\t)}\right)^{24}=q^{-1}-24+276q-2048q^2+11202q^3-49152q^4+184024q^5-\cdots$$
is a Hauptmodul for $X_0(2)$, where $\eta(\t)=q^{1/24}\prod_{n=1}^\i(1-q^n)$ is the Dedekind eta function. Using this Hauptmodul, we can compute the Fourier coefficients of $f_{22,3}^{(11)}=J_{2,3}+J_{2,1}$, the 11-plicate of $f_{22,3}$:
$$f_{22,3}^{(11)}=q^{-3}+q^{-1}+33882q-1845248q^2+43446018q^3-648265728q^4+7171488865q^5-\cdots\in M_0^{!,\infty}(2).$$
Again, $f_{22,3}$ and $f_{22,3}^{(11)}$ share the same principal part and satisfy the congruence $$f_{22,3}\equiv f_{22,3}^{(11)}\pmod{11}.$$
\end{exam}

Using properties of the Hecke operator, we prove several more congruences that Fourier coefficients of reduced row echelon bases of $M_0^{!,\i}(\G)$ of non-zero genus satisfy. Theorem \ref{Tcong} is one of them, which gives quite a strong congruence of Fourier coefficients of modular forms. The congruence holds for arbitrary prime powers, including powers of prime divisors of $N$.
For $m\geq 1$, we write $f_{\G,m}$ as
\begin{equation}\label{fF}f_{\G,m}(\t)=q^{-m}+\sum_{l=1}^{g} a_{\G}(m,-l)q^{-l}+\sum_{n\geq 1} a_{\G}(m,n)q^n\in M_0^{!,\infty}(\G).\end{equation} 

\begin{thm}\label{Tcong}
Let $\G=N+S$ and suppose $X(\G)$ is of genus $g\geq 1$.  Assume that $i\i$ is not a Weierstrass point of $X(\G)$. If a prime $p>g$, then for any positive integers $r$, $n$ and $m>g$ 
 with $p\nmid m$, $p\nmid n$, we have
\begin{equation}\label{eq:cong}a_{\G}(mp^r,n)+\sum_{l=1}^g a_{\G}(m,-l)a_{\G}(lp^r,n)\equiv 0\pmod{p^r},\end{equation}
where these coefficients are all integers.
\end{thm}
We can derive many special congruences from Theorem \ref{Tcong}. For example, since $a_{11}(8,-1)=a_{11}(19,-1)=0$ and all the Fourier coefficients $a_{\G}(m,n)$ are integers when $X(\G)$ has genus $1$ by our construction, we  have from Theorem \ref{Tcong} that
$$a_{11}(8p^r,n)\equiv 0\ (\mathrm{mod}\ {p^r})\ \mathrm{if}\ p\neq 2, p\nmid n
\quad \mathrm{and} \quad
a_{11}(19p^r,n)\equiv 0\ (\mathrm{mod}\ {p^r})\ \mathrm{if}\ p\neq 19, p\nmid n.$$
Hence we find that $\displaystyle{a_{11}(152,n)\equiv 0\pmod{152}}$ for any $n$ coprime to $38$.

In order to establish a generalization of \eqref{gj} to $J_{\G,n}$ for arbitrary $\G$, we first prove a duality relation between coefficients of weakly holomorphic modular forms, a so-called Zagier duality. Let $M_k(\G)$ ($S_k(\G)$, resp.) denote the space of modular forms (cusp forms, resp.) of weight $k$ on $\G$.
Since $\dim(S_2(\G))=g$ and $i\i$ is not a Weierstrass point on $X(\G)$, there exists a unique basis $\{h_{\G,-l}\}_{1\leq l \leq g}$ of $S_2(\G)$ with Fourier expansions
\begin{equation}\label{mgf}h_{\G,-l}=q^{l}+\sum_{n=g+1}^\infty b_{\G}(-l,n)q^n,\end{equation}
when $g>0$. Using linear combinations of products of $f_{\G,m}$ and $h_{\G,-l}$ for $l=1,2,\cdots,g$, one can construct a unique modular form in $S_2^{!,\infty}(\G)$ with Fourier expansion $h_{\G,n}= q^{-n} + O(q^{g+1})$ for each integer $n\geq-g$. When $g=0$, one can take $h_{\G,n}(\t)$ by differentiating $J_{\G,n}(\t)$ with respect to $\t$ and then normalizing it.
These naturally form a basis for $S_2^{!,\infty}(\G)$.  We write for each $n\geq -g$
\begin{equation}\label{gF}
h_{\G,n}(\t)=q^{-n}+\sum_{m\geq g+1} b_{\G}(n,m)q^m\in S_2^{!,\infty}(\G). \end{equation}
The Fourier coefficients of $f_{\G,m}(\t)$ and $h_{\G,n}(\t)$ satisfy the following duality condition. 

\begin{thm}\label{Tgrid} 
Let $\G=N+S$ and $g$ denote the genus of the modular curve $X(\G)$.
Assume that $i\i$ is not a Weierstrass point on $X(\G)$.
Let $m\geq g+1$ and $n\geq -g$ be integers. Then we have
$$a_{\G}(m,n)=-b_{\G}(n,m).$$
\end{thm}

\begin{rmk}\label{iw}If $i\i$ is a Weierstrass point of $X(\G)$ and $1=n_1<n_2<\cdots<n_g\leq 2g-1$ are the $g$ gaps at $i\i$, then $S_2(\G)$ is generated by the cusp forms in the form of $q^{n_j}+O(q^{2g})\ (1\leq j\leq g)$. So one still can construct reduced row echelon bases for $M_0^{!,\i}(\G)$ and $S_2^{!,\i}(\G)$ which satisfy a Zagier duality like in Theorem \ref{Tgrid}. (See \cite{Treneer} for the existence of such bases in detail.)
\end{rmk}
As in the proof of \eqref{gj} in \cite{AKN}, duality of this kind was frequently used to find an explicit formula for the generating function of basis elements for $M_k^!(\G)$ with the level of genus zero. Recently, Jenkins and Molnar \cite{JM} extended the strategy to prime levels of nonzero genus.  For prime level of genus 1, i.e., when $p\in\{11,17,19\}$, they derived explicit formulas for generating functions of the reduced row echelon bases for $M_k^{!,\i}(p)$ of arbitrary even integer weight $k$ using the duality they found.  Employing the duality relation in Theorem \ref{Tgrid}, we establish an explicit representation of the generating function of $f_{\G,m}(\t)$ on any $\G$.
\begin{thm}\label{genf} Assume the modular curve $X(\G)$ has genus $g$ and $i\i$ is not a Weierstrass point of $X(\G)$. If $\t_1,\tau_2\in\mathbb H$, then
\begin{align}\label{gfg}
\sum_{n=0}^\infty f_{\G,n}(\t_1)q_2^n&=-\frac{1}{2\pi i}\frac{f'_{\G,g+1}(\t_2)}{f_{\G,g+1}(\t_2)-f_{\G,g+1}(\t_1)}+\frac{A(\t_1,\t_2)}{f_{\G,g+1}(\t_2)-f_{\G,g+1}(\t_1)},
\end{align}
where \begin{align*}\nonumber
A(\t_1,\t_2)&=\sum_{\ell=-g}^{g+1} (1-\ell) a_\G(g+1,-\ell)h_{\G,\ell}(\t_2)+\sum_{\ell=1}^g f_{\ell+g+1}(\t_1)h_{\G,-\ell}(\t_2)\\
&\qquad +\sum_{j=g+1}^{2g}\sum_{\ell=j-g}^{g}a_\G(g+1,\ell-j)f_j(\t_1)h_{-\ell}(\t_2).
\end{align*}
\end{thm}

 \begin{rmk} 
(1) With the bases alluded to in Remark \ref{iw}, one can also establish an explicit form of the generating function of the basis of $M_0^{!,\i}(\G)$ when $i\i$ is a Weierstrass point of $X(\G)$.

(2) When $X(\G)$ has genus $0$, $f_{\G,1}(\t)$ is the Hauptmodul and $A(\t_1,\t_2)=0$. Thus \eqref{gfg} is a direct generalization of \eqref{gj}.
 \end{rmk}

 In \cite[Theorem 1.1 (i)]{BKLOR}, Bringmann and {\it{et al.}} showed that the generating function of the harmonic Maass functions $J_{N,n}$ is completed to a weight 2 polar harmonic Maass form of level $N$ (that is, a pole is allowed in $\mathbb H$). In Theorem \ref{genj} below shows that the non-holomorphic part of the weight 2 polar harmonic Maass form can be canceled by some linear combination of weight 2 cusp forms and harmonic Maass functions, while the resulting holomorphic part gives the generating function of  $f_{\G,n}$'s. 
\begin{thm}\label{genj}  Assume the modular curve $X(\G)$ has genus $g$ and $i\i$ is not a Weierstrass point of $X(\G)$.  If $q_1=e^{2\pi i \t_1}$ and $q_2=e^{2\pi i \t_2}$ for $\t_1,\tau_2\in\mathbb H$, then
\begin{equation}
\sum_{n=0}^\i J_{\G,n}(\t_1) q_2^n-\sum_{l=1}^g h_{\G,-l}(\t_2)J_{\G,l}(\t_1)=\sum_{n=0}^\infty f_{\G,n}(\t_1)q_2^n.
\end{equation}
\end{thm}

The rest of the paper is organized as follows. In Section 2, we introduce Niebur-Poincar\'e series and derive different representations of its replicates. We also present basic properties of Atkin-Lehner involutions. In Section 3, we use them to prove the expansion formulas and compression formulas for $J_{\G,n}$ and eventually prove Theorem \ref{p-pli}. In Section 4, we construct a Hecke operator for harmonic Maass functions and prove the replication formula for  $J_{\G,n}$ in Theorem \ref{replic}, which implies that $J_{\G,n}$'s form a Hecke system.  In Section 5, we start to investigate arithmetic properties of weakly holomorphic modular functions of arbitrary levels with Zagier duality in Theorem \ref{Tgrid}, as it is needed in the subsequent sections. In Section 6, we discuss the action of the Hecke operator on $M_0^{!,\infty}(\G)$ and prove Theorems \ref{tpm} and \ref{heckepres}.  In Sections 7 and 8, we further discuss replicability of weakly holomorphic modular functions of higher genera and  prove several congruences that the Fourier coefficients of weakly holomorphic modular functions satisfy, including Theorems \ref{ffp} and \ref{Tcong}. Finally in Section 9, we establish explicit formulas of the generating functions of $f_{\G,n}(\t)$ and $J_{\G,n}(\t)$  proving Theorems \ref{genf} and \ref{genj}.

\section{Niebur-Poincar\'e series on $\G$ and $\G^{(p)}$}
Harmonic Maass forms of weight $k$ for $\G$ are smooth complex valued functions on $\mathbb H$ that are $\G$-invariant. They are eigenfunctions of the weight $k$ hyperbolic Laplacian with eigenvalue $0$ and have at most exponential growth at all the cusps. Moreover, they naturally decompose into holomorphic parts and non-holomorphic parts. The holomorphic parts are nowadays called mock modular forms and the non-holomorphic parts hide the companions of the holomorphic parts, called shadows of the mock modular forms. The shadows are cusp forms of weight $2-k$ that can be revealed via the antilinear differential operator $\xi_k:=2i({\rm{Im}}\t)^k \overline{\frac{\partial}{\partial\bar \t}}$. 

Suppose $\G_\i:=\{\pm\sm 1&t\\0&1\esm:t\in\mathbb Z\}$ is the subgroup of translations of $\G$. We define the Niebur-Poincar\'e series for positive integers $n$ and ${\rm Re} \, (s) > 1$ by
\begin{equation}\label{gp}
G_{\G,-n}(\tau, s):=2\pi \sqrt{n}\sum_{\g\in \Gamma_\infty \backslash \Gamma}f_n(\g\t),
\end{equation} 
where \begin{equation}\label{fn}f_n(\t):=e(-n {\rm Re}(\tau)) (n{\rm Im}(\tau))^{\frac12}I_{s-\frac{1}{2}}(2\pi n{\rm{Im}}(\t)).\end{equation}
Here $e(x)=\exp(2\pi i x)$ and $I_{s-\frac12}$ denotes the $I$-Bessel function.
These $\G$-invariant functions on $\H$ satisfy that 
\begin{equation}\label{npdelta}\Delta G_{\G,-n}(\t,s)=s(1-s)G_{\G,-n}(\t,s),\end{equation}
where $\Delta$ is the weight $0$ hyperbolic Laplacian.
As each $G_{\G,-n}(\t,s)$ has an analytic continuation to ${\rm Re} \,(s) > 1/2$, we obtain an infinite family of harmonic Maass functions $\{G_{\G,-n}(\t,1): n\in \mathbb N\}$, which decay like cusp forms at cusps inequivalent to $i\i$. (See \cite[p.~98]{BFOR}.) 
Now we write for each positive integer $n$, 
\begin{equation}\label{jnn}j_{\G,n}(\t):=G_{\G,-n}(\t,1),
\end{equation}
the harmonic Maass function with the principal part $q^{-n}$ introduced earlier in \eqref{jnnfour}.

For convenience, we let 
\begin{equation}\label{fns}
F_{\G,n}(\t,s):=\sum_{\g\in\Gamma_\infty\backslash\G}f_n(\g\t).
\end{equation}
As $\G=N+S$ is the subgroup of ${\rm PSL}_2(\R)$ generated by $\G_0(N)$ and the Atkin-Lehner involutions $W_{Q,N}$ for all $Q\in S$,  $\G$ has the decomposition
\begin{equation}\label{decg}
\Gamma=\Gamma_0(N)\bigcupdot_{Q\in S}\Gamma_0(N)W_{Q,N}.
\end{equation}
Thus we may write $F_{\G,n}(\t,s)$ as
\begin{equation}\label{fgnpt}
F_{\G,n}(\t,s)=\sum_{\g\in \Gamma_\infty\backslash\Gamma_0(N)}f_n(\g(\t))+\sum_{Q\in S}\sum_{\g\in \Gamma_\infty\backslash\Gamma_0(N)}f_n(\g W_{Q,N}(\t)).  \end{equation}

Meanwhile, for any prime $p$, $\G^{(p)}=N^{(p)}+S^{(p)}$, where $N^{(p)}=N/(p,N)$ and $S^{(p)}$ is the set of all $Q\in S$ which divide $N^{(p)}$.
In order to find an explicit representation of $F_{\G^{(p)},n}$, we need to find coset decompositions of $\Gamma_\infty\backslash\Gamma^{(p)}$. We start with finding a coset decomposition of $\Gamma_\infty\backslash\Gamma_0(M)$ when $N=pM$.

\begin{lem}\label{dec}
Let $p$ be a prime divisor of $N$ and write $N=pM$.
For $\gamma=\left(\begin{smallmatrix} a&b\\c&d\end{smallmatrix}\right)\in\Gamma_0(N)$,
we let $\alpha_\gamma=\left(\begin{smallmatrix} a&pb\\ \frac{c}{p}&d\end{smallmatrix}\right)\in\Gamma_0(M)$ and
$\beta_\gamma=\gamma\mu_M$ where $\mu_M=\left(\begin{smallmatrix} x&y\\ Mz&pw\end{smallmatrix}\right)\in\Gamma_0(M)$ for some $x,y,z,w\in\mathbb Z$.
Let 
\begin{equation}\label{AB}\mathcal A=\left\{ \Gamma_\infty\alpha_\gamma \, |\, \gamma\in\Gamma_0(N)\right\}\ {\rm{and}}\ \mathcal B=\left\{ \Gamma_\infty\beta_\gamma \, |\, \gamma\in\Gamma_0(N)\right\}\end{equation}
be subsets of $\Gamma_\infty\backslash\Gamma_0(M)$. Then
\begin{equation}\label{decom} \Gamma_\infty\backslash\Gamma_0(M)=
\left\{\begin{array}{ll}
\mathcal A\, \dotcup \,\mathcal B,
& \hbox{ if } p\nmid M,\\
\mathcal A,
&  \hbox{ if } p\mid M.\end{array}\right.\end{equation}
\end{lem}

\begin{proof}
Suppose $\Gamma_\infty\alpha_{\gamma}=\Gamma_\infty\beta_{\gamma'}$
for some $\gamma=\sm a&b\\c&d\esm,\gamma'=\sm a'&b'\\c'&d'\esm\in\Gamma_0(N)$.
Then $\gamma'\mu_M=\sm 1&t\\0&1\esm \alpha_\gamma$ for some $t\in\Z$, i.e.,
$$\begin{pmatrix} a'x+b'Mz&a'y+b'pw\\ c'x+d'Mz& c'y+d'pw\end{pmatrix}=\begin{pmatrix} a+\frac{ct}{p}&pb+dt\\ \frac{c}{p}& d\end{pmatrix},$$
which implies that $d=c'y+d'pw$. However this cannot occur, because $p|c'y+d'pw$ while $p\nmid d$.
Thus $\mathcal A\cap\mathcal B=\emptyset$. Next, for a given 
$\Gamma_\infty\delta\in\Gamma_\infty\backslash\Gamma_0(M)$ with $\delta=\sm e&f\\g&h\esm\in \G_0(M)$, we show that either $\Gamma_\infty\delta\in \mathcal A$ or $\Gamma_\infty\delta\in \mathcal B$.  
In the case when $p\nmid h$, $h t\equiv -f \pmod{p}$ has a unique solution $t=t_0$ modulo $p$.
If we let $a=e+gt_0$, $b=\frac{f+h t_0}{p}$, $c=pg$, and $d=h$, then we have
\begin{align*} \begin{pmatrix} a&pb\\ \frac{c}{p}& d\end{pmatrix} &=\begin{pmatrix} e+gt_0& f+h t_0\\ g & h\end{pmatrix} \\ &=\begin{pmatrix} 1& t_0\\ 0 & 1\end{pmatrix}\begin{pmatrix} e& f\\ g & h\end{pmatrix}
\end{align*}
and $ad-bc=eh-gf=1$. Hence  $\Gamma_\infty\delta=\Gamma_\infty \alpha_\gamma$ for $\gamma=\sm a&b\\c&d\esm \in\Gamma_0(N)$.
In the case when $p|h$, we let $\gamma=\delta\omega^{-1}$.
Then the $(2,1)$-component of $\gamma$ is equal to $gpw-h Mz\equiv 0\pmod{N}$, because $g\mid M$ and $p|h$. Hence $\gamma\in\Gamma_0(N)$, and thus 
$\Gamma_\infty\delta=\Gamma_\infty \beta_\gamma$. When $p\mid M$, the case when $p\mid h$ does not arise. Therefore, we have proved the lemma. 
\end{proof}

By Lemma \ref{dec} and \eqref{decg}, for $\G=N+S$ with $N=pM$, we have the following decomposition:
\begin{equation}\label{decgi}
\G_\infty\backslash\Gamma^{(p)}=
\begin{cases}
\left(\mathcal A\bigcupdot \mathcal B\right) \bigcupdot_{Q\in S^{(p)}}\left(\mathcal A W_{Q,M} \bigcupdot \mathcal B W_{Q,M}\right),&\hbox{ if } p\nmid M, \\
\mathcal A \bigcupdot_{Q\in S^{(p)}}\mathcal A W_{Q,M},&\hbox{ if } p\mid M.
\end{cases}
\end{equation}

Therefore by \eqref{fgnpt} and \eqref{AB}, if $p\nmid M$, then we have
\begin{align}\label{fgpnpt0}\nonumber
F_{\G^{(p)},n}(\t)
=& \sum_{\Gamma_\infty\gamma\in\Gamma_\infty\backslash\Gamma_0(N)}f_n(\Gamma_\infty\alpha_\gamma (\t))+ \sum_{\Gamma_\infty\gamma\in\Gamma_\infty\backslash\Gamma_0(N)}f_n(\Gamma_\infty\beta_\gamma (\t))\\ 
&+\sum_{Q\in S^{(p)}}\left(\sum_{\Gamma_\infty\gamma\in\Gamma_\infty\backslash\Gamma_0(N)}f_n(\Gamma_\infty\alpha_\gamma W_{Q,M}(\t))+ \sum_{\Gamma_\infty\gamma\in\Gamma_\infty\backslash\Gamma_0(N)}f_n(\Gamma_\infty\beta_\gamma W_{Q,M} (\t))\right),
\end{align}
and if $p\mid M$, then 
\begin{equation}\label{fgpnptm0}
F_{\G^{(p)},n}(\t)
=\sum_{\Gamma_\infty\gamma\in\Gamma_\infty\backslash\Gamma_0(N)}f_n(\Gamma_\infty\alpha_\gamma (\t))+\sum_{Q\in S^{(p)}}\sum_{\Gamma_\infty\gamma\in\Gamma_\infty\backslash\Gamma_0(N)}f_n(\Gamma_\infty\alpha_\gamma W_{Q,M}(\t)).
\end{equation}

We discover more coset decompositions of $\Gamma_\infty\backslash\Gamma_0(M)$ when $N=pM$.
\begin{lem}\label{dec2}
Let $p$ be a prime divisor of $N$ and write $N=pM$. For $\gamma\in\Gamma_0(N)$, we define $\alpha_\gamma$ and $\beta_\gamma$ as in Lemma \ref{dec}.  Let $\omega_M=\left(\begin{smallmatrix} px&y\\ Mz&w\end{smallmatrix}\right)\in\Gamma_0(M)$ for some $x,y,z,w\in\mathbb Z$.
If $p\nmid M$, then
\begin{equation}\label{dem} \Gamma_\infty\backslash\Gamma_0(M)=
 \Gamma_\infty\backslash\Gamma_0(N)\, \bigcupdot \,\{ \Gamma_\infty\alpha_\gamma \omega_M\,|\,\gamma\in\Gamma_0(N)\}.
\end{equation}
We also decompose  $\Gamma_\infty\backslash\Gamma_0(M)$ as
\begin{equation}\label{dem2} \Gamma_\infty\backslash\Gamma_0(M)=
 \Gamma_\infty\backslash\Gamma_0(N)\, \bigcupdot \,\{  \bigcupdot_{j=0}^{p-1}\Gamma_\infty\beta_\gamma T^j\,|\,\gamma\in\Gamma_0(N)\},
\end{equation}
where $T=\sm 1& 1\\0&1\esm$.
 If $p\mid M$,  then 
\begin{equation}\label{dem3} \G_\infty\backslash\G_0(M)=\bigcupdot_{j=0}^{p-1}\{ \G_\infty\gamma \delta_j\,|\,\gamma\in\G_0(N)\}, \end{equation}
where $\delta_j=\sm a&b\\jcM&d\esm\in\G_0(M)$ such that $\delta_j\equiv I\pmod p$ and $c\equiv c_0\nequiv 0\pmod p$ where $c_0$ is a fixed intger with $(c_0,N)=1$ for $j=0,1,2,\cdots,p-1$.

\end{lem}
\begin{proof} One can easily check that the $(2,1)$-component of $\alpha_\gamma\omega_M$ is not divisible by $p$, hence the two sets on the right-hand side of \eqref{dem} are disjoints.
For a given 
$\delta=\sm e&f\\g&h\esm\in \G_0(M)$ with $p\nmid g$, we find that the $(2,2)$-component of $\delta\omega_M^{-1}$ is equal to $-gy+ph x$ which is not divisible by $p$.
Thus from the proof of Lemma \ref{dec}, $\G_\infty\delta\omega_M^{-1}\in \mathcal A$. That is, $\G_\infty\delta\omega_M^{-1}=\G_\infty\alpha_\gamma$ for some $\gamma\in\G_0(N)$. Hence $\G_\infty\delta=\G_\infty\alpha_\gamma\omega_M$, and \eqref{dem} holds.

For  \eqref{dem2}, first note that $\G_\i\g'\neq \G_\i\beta_\g T^j$ for any $\g,\g'\in \G_0(N)$ and $j$, because $\beta_\g T^j$ cannot be contained in $\G_0(N)$. Next, observe that if $i\neq j$, then $\G_\i\beta_\g T^j\neq\G_\i\beta_{\g'} T^j$, because otherwise $\G_\i\beta_\g T^i=\G_\i\beta_{\g'}$ for some $\g,\g'\in \G_0(N)$ and a positive $i$, but it is a contradiction to the fact that (2,2)-component of $\beta_{\g'}$ is divisible by $p$ while that of $\beta_{\g} T^i$ ($i\neq 0$) is not. So all the sets in the right-hand side of \eqref{dem2} are disjoints. Now, let $\delta=\sm e&f\\g&h\esm\in \G_0(M)$ with $p\nmid g$ given. Then from the proof of Lemma \ref{dec}, we see that if $p\mid h$, then $\G_\infty\beta_\g=\G_\i\delta$ for some $\g\in \G_0(N)$. If $p\nmid h$, we verify that the $(2,1)$-component of $\delta T^j \mu_M^{-1}$ is equal to 
\begin{equation}\label{pgw}
pgw-(h+gj)Mz.
\end{equation}
Since $p\nmid g$ and $p\nmid h$, there exists a unique $j$ mod $p$ such that \eqref{pgw} is divisible by $N$, and hence $\delta T^j \mu_M^{-1}\in \G_0(N)$ for the $j$. This implies that $\G_\infty\delta=\G_\infty\beta_\gamma T^j$ for some $\g\in \G_0(N)$ and $j$, which ends the proof of \eqref{dem2}.

Lastly, we assume $p\mid M$ and prove \eqref{dem3}. 
Let $\delta_j=\sm a&b\\jcM&d\esm$ and $\delta_j'=\sm a'&b'\\jc'M&d'\esm$.
Then the $(2,1)$-component of $\delta_j'\delta_j^{-1}$ is equal to $jc'dM-jcd'M$ and is divisible by $N$, because $d\equiv d'\equiv 1\pmod p$ and $c\equiv c'\equiv c_0\pmod p$ imply that $jc'd-jcd'\equiv (d-d')jc_0\equiv 0\pmod p$. Thus $\delta_j'\delta_j^{-1}\in\G_0(N)$ and \eqref{dem3} is independent of the choice of $\delta_j$'s.
Using the same argument, we can show that $\delta_i\delta_j^{-1}=\delta_{i-j}$ if $i> j$.
Thus in order to show the sets on the right-hand side of \eqref{dem3} are disjoint, it suffices to prove $\G_\infty\gamma\delta_j\neq \G_\infty\gamma$ if $j\neq 0$. But this is obviously true, because $\delta_j\notin \G_0(N)$.
Finally, let $\delta=\sm e&f\\g&h\esm\in \G_0(M)$ with $N\nmid g$ given. Since $(e,N)=1$, there exists an $e'$ such that $ee'\equiv 1\pmod N$, and hence we can find a $\gamma_1\in\G_0(N)$ whose $(1,1)$-component is $e'$ and  $\gamma_1\delta\equiv \sm 1 &k\\ 0 & 1\esm \pmod{p}.$
Thus $T^{-k}\gamma_1\delta\equiv I\pmod{p}$ and it is easy to check its $(2,1)$-component is a multiple of $M$ but not of $N$. 
Accordingly, we can take $\delta_j=T^{-k}\gamma_1\delta$ for some $j\neq 0$,  from which we derive that $\G_\infty\delta=\G_\infty\gamma\delta_j$ by setting $\gamma=\gamma_1^{-1}T^k$. Therefore, we have proved \eqref{dem3}.
\end{proof}

Next, we investigate a relation between $\G_\i\backslash \G_0(N)$ and $\G_\i\backslash \G_0(M)$.  For positive integers $M$ and $N$, we let  $\Gamma_0(N,M)=\{\sm a&b\\c&d\esm\in{\rm SL}_2(\Z)\,|\, N|c, M|b\}$ and for $n\in \mathbb N$, let $\Gamma_\infty^{(n)}=\{\sm 1&nt\\0&1\esm\,|\, t\in\Z\}$.
\begin{lem}\label{11}
Let $p$ be a prime divisor of $N$ and write $N=pM$.
For $\gamma=\left(\begin{smallmatrix} a&b\\c&d\end{smallmatrix}\right)\in\Gamma_0(N)$,
we let $\alpha_\gamma=\left(\begin{smallmatrix} a&pb\\ \frac{c}{p}&d\end{smallmatrix}\right)\in\Gamma_0(M)$ as in Lemma \ref{dec}.
Then the map $\psi:\Gamma_\infty\backslash \Gamma_0(N)\to\Gamma_\infty^{(p)}\backslash\Gamma_0(M,p)$ defined by $\psi(\Gamma_\infty\gamma)=\Gamma_\infty^{(p)}\alpha_\gamma$ is a 1-1 correspondence.
\end{lem}
\begin{proof} Suppose $\Gamma_\infty^{(p)}\alpha_\gamma=\Gamma_\infty^{(p)}\alpha_{\gamma'}$, where $\gamma=\sm a&b\\c&d\esm$ and $\gamma'=\sm a'&b'\\c'&d'\esm$.
Then
\begin{align*} \begin{pmatrix} a'&pb'\\ \frac{c'}{p}& d'\end{pmatrix} &=\begin{pmatrix} 1& pt\\ 0 & 1\end{pmatrix}\begin{pmatrix} a& pb\\ \frac{c}{p} & d\end{pmatrix} \\ &=\begin{pmatrix} a+ct& p(b+dt)\\ \frac{c}{p} & d\end{pmatrix}.
\end{align*}
Thus $a'=a+ct$, $b'=b+dt$, $c'=c$, and $d'=d$, which implies that $\Gamma_\infty\gamma=\Gamma_\infty\gamma'$. That is, $\psi$ is injective. For any $\delta=\sm a&pb\\ Mc&d\esm\in\Gamma_0(M,p)$, let $\gamma=\sm a&b\\ pMc&d\esm\in\Gamma_0(N)$. Then $\psi(\Gamma_\infty\gamma)=\Gamma_\infty^{(p)}\alpha_\gamma=\Gamma_\infty^{(p)}\delta$, which proves that $\psi$ is surjective.
\end{proof}

Before we end this section, we list several properties of Atkin-Lehner involutions  that will be used later.  If $e$ is an exact divisor of $N$, that is, $e|N$ and $(e,N/e):= gcd(e, N/e)=1$, then we write $e\parallel N$. For any $e\parallel N$, the Atkin-Lehner involutions are given by the matrices
 $$W_{e,N}=\begin{pmatrix} ex&y\\Nz&ew \end{pmatrix},$$ 
where $\det (W_{e,N})=e$ and $x,y,z,w \in\Z$. These matrices $W_{e,N}$ normalize $\G_0(N)$ and $\G/\G_0(N)$ is an elementary abelian 2-group. (\cite[Lemma 9]{AL})

\begin{lem}\label{ko}
Let p be a prime divisor of N and write N = pM.

\begin{enumerate}
\item \label{ko1} Let $p\nmid Q$, and take $W_{Q,M}=\begin{pmatrix} Qx&y\\Mz&Qw\end{pmatrix}$ such that $p\mid z$.
Then for each $j\in\Z$, there exists a unique $i$ modulo $p$ such that
\begin{equation*}
\begin{pmatrix} 1&j\\0&p\end{pmatrix}W_{Q,M}=W_{Q,pM}\begin{pmatrix} 1&i\\0&p\end{pmatrix}
\end{equation*}
for some $W_{Q,pN}$.

\item \label{ko2} Recall that $V_p=\begin{pmatrix} p&0\\0&1\end{pmatrix}$.  If $p\nmid Q$, then
\begin{equation*}
V_pW_{Q,pN}=W_{Q,N}V_p.
\end{equation*}

\item\label{ko3} Assume $p\mid M$. If $p\mid Q$, then there exists some $\delta_j\in\G_0(M)$ such that $\delta_j\equiv I\pmod p$ and
\begin{equation*}
W_{Q,M}T^j=\delta W_{Q,M},
\end{equation*}
where $T=\begin{pmatrix} 1&1\\0&1\end{pmatrix}$.
Moreover, $\delta_j\notin\G_0(N)$ if $j\neq 0.$

\item\label{ko4} Assume $p^\nu \| M$. If $p\nmid Q$, then
\begin{equation*}
W_{p^{\nu+1}Q,N}=W_{p^\nu Q,M}V_p,\hbox{ and }
V_pW_{p^{\nu+1}Q,N}=W_{p^\nu Q,M}\begin{pmatrix} p&0\\0&p\end{pmatrix}.
\end{equation*}

\end{enumerate}
\end{lem}

\begin{proof} Since others are found in Lemmas 2.5, 2.6 and 2.8 of \cite{Koike},  we prove only \eqref{ko3}.
By direct computation, we have
\begin{equation}\label{wtw}
W_{Q,M}T^j W_{Q,M}^{-1}=\begin{pmatrix}Qxw - jMxz - \frac{M}{Q}yz & jQx^2\\ j\left(-\frac{M}{Q}z^2\right)M& Qxw + jMxz - \frac{M}{Q}yz\end{pmatrix}\in\G_0(M).
\end{equation}
Since $\det(W_{Q,M})=Q$, $Qxw-\frac{M}{Q}yz=1$. Thus the first part of \eqref{ko3} is true.  In addition, since $p\nmid  \left(-\frac{M}{Q}z^2\right)$, the rest holds as well.
\end{proof}

\section{Proof of $p$-plication formula for $J_{\G,n}(\t)$}

In order to prove the {\it{$p$-plication formula}} for Hauptmoduls \cite[Theorem 1.1]{Koike}, Koike divided all $\G=N+S$ when $N=pM$ for a prime $p$ into the following five cases: 
\begin{itemize}
\item (Case 1) $p\nmid M$ and $p\in S$,
\item (Case 2) $p\nmid M$ and $p\nmid Q$ for any $Q\in S$,
\item (Case 3) $p\nmid M$ and there exists an exact divisor $Q_0\neq1$ of $M$ such that $pQ_0\in S$,
\item (Case 4) $p\mid M$ and there exists some  $pQ_0\in S$,
\item (Case 5) $p\mid M$ and $p\nmid Q$ for any $Q\in S$.
\end{itemize}
He then proved the {\it{expansion formula}} \cite[Theorem 4.1]{Koike} and the {\it{compression formula}} \cite[Theorem 3.1]{Koike} for each case, and combined them into one formula, the $p$-plication formula.
We also need to consider each of five cases separately as the coset decompositions of $\G_\i\backslash \G^{(p)}$ to define $j^{(p)}_{\G,n}$ vary depending on $p$ and $S$. %
We first prove expansion formulas of $j_{\G,n}$ for all five cases.
\begin{prop}[{\it{Expansion formulas}}]  \label{TnN} Let $p$ be a prime divisor of $N$ and write $N=pM$. 
Then the following hold:
\begin{itemize}
\item {\rm (Case 1)} $j_{\G^{(p)},n}(\t)+j_{\G^{(p)},n}(p\t)=j_{\G,pn}(\t)+j_{\G,n}(\t).$
\item {\rm (Case 2)} $j_{\G^{(p)},n}(p\t)=j_{\G,pn}(\t)+j_{\G,n}(\t)|{W_{p,N}}.$
\item {\rm (Case 3)} $j_{\G^{(p)},n}(\t)+j_{\G^{(p)},n}(p\t)|{W_{Q_0,N}}=j_{\G,n}(\t)+j_{\G,pn}(\t)|{W_{Q_0,N}}.$
\item {\rm (Case 4)} $j_{\G^{(p)},n}(p\t)+j_{\G^{(p)},n}(\t)|{W_{Q_0,M}}=j_{\G,pn}(\t).$
\item {\rm (Case 5)} $j_{\G^{(p)},n}(p\t)=j_{\G,pn}(\t).$
\end{itemize}
\end{prop}
\begin{rmk}  Theorem 1.1 (3) of \cite{BKLOR} gives a special case of Case (5) when $\G=\G_0(N)$. (It is not correct as stated. It is correct when a prime $p$ satisfies $p^2\mid N$ (see \cite{BKLOR-c}.))  
\end{rmk}

\begin{proof}[Proof of Proposition \ref{TnN}]
Instead of $j_{\G,n}(\t)$, it suffices to prove the results for $F_{\G,n}(\t,s)$ and use analytic continuation. We begin by pointing out that among the sets of exact divisors of $N$ that generate the same $\G=N+S$, we take $S$ the largest. Also, we note that $p\nmid Q$ for any $Q\in S^{(p)}$. 

We first prove the last two cases when $p\mid M$.
If $p\mid M$, then it follows from \eqref{fgpnptm0}  and properties of Atkin-Lehner involutions, including \eqref{ko2} of Lemma \ref{ko} that 
\begin{align}\label{fgpnptm}
F_{\G^{(p)},n}(p\t)=&\sum_{G\in\Gamma_\infty\backslash\Gamma_0(N)}f_n(pG(\t))+\sum_{Q\in S^{(p)}}\sum_{G\in\Gamma_\infty\backslash\Gamma_0(N)}f_n(pGW_{Q,N}(\t)), \end{align}
because for any $\g\in\G_0(N)$, $\alpha_\gamma V_p = V_p\gamma$ and for any integer  $j$, $V_pT^j=T^{pj}V_p$.

$\cdot$\ Case (5):  Since $S=S^{(p)}$, \eqref{fgpnptm} is simply a statement of
 $$F_{\G^{(p)},n}(p\t)=F_{\G,pn}(\t),$$ which proves the expansion formula.

$\cdot$\ Case (4): Since $p\mid M$ and $pQ_0\parallel N$, $p\mid Q_0$ and $Q_0\parallel M$. By \eqref{fgpnptm0} in the first equality and by using the second identity in \eqref{ko4} of Lemma \ref{ko} in the second equality below, we have
\begin{align}\label{fgpnw} \nonumber
F_{\G^{(p)},n}(\t)|W_{Q_0,M}=& \sum_{\G_\infty\gamma\in\Gamma_\infty\backslash\Gamma_0(N)}f_n(\G_\infty\alpha_\gamma W_{Q_0,M}(\t))+\sum_{Q\in S^{(p)}}\left(\sum_{\G_\infty\gamma\in\Gamma_\infty\backslash\Gamma_0(N)}f_n(\G_\infty\alpha_\gamma W_{Q,M} W_{Q_0,M}(\t))\right)\\ 
=& \sum_{\G_\infty\gamma\in\Gamma_\infty\backslash\Gamma_0(N)}f_n(\G_\infty\alpha_\gamma V_pW_{pQ_0,N}(\t))+\sum_{Q\in S^{(p)}}\left(\sum_{\G_\infty\gamma\in\Gamma_\infty\backslash\Gamma_0(N)}f_n(\G_\infty\alpha_\gamma V_p W_{pQ_0,N} W_{Q,M} (\t))\right)\\ \nonumber
=& \sum_{G\in\Gamma_\infty\backslash\Gamma_0(N)}f_n(pGW_{pQ_0,N}(\t))+\sum_{Q\in S^{(p)}}\left(\sum_{G\in\Gamma_\infty\backslash\Gamma_0(N)}f_n(pG W_{pQ_0,N} W_{Q,M} (\t))\right).
\end{align}
The last equality is obtained as before.  If $Q'\in S$ and $p\nmid Q'$, then $Q'\in S^{(p)}$. If $Q'\in S$ and $p\mid Q'$, then either $Q'=pQ_0$ or $W_{Q',N}W_{pQ_0,N}= W_{Q,N}$ for some $Q\in S^{(p)}$. Here, as $Q$ satisfies both $Q\| M$ and $Q\| N$, $W_{Q,M}$ or $W_{Q,N}$ can be used interchangeably, whenever applicable.  Hence by combining \eqref{fgpnptm} and \eqref{fgpnw} together, we find that
$$F_{\G^{(p)},n}(p\t)+F_{\G^{(p)},n}(\t)|W_{Q_0,M}=F_{\G,pn}(\t).$$
Hence we have the expansion formula.

From now, we assume $p\nmid M$. 
%
Since $\beta_\gamma V_p=\gamma W_{p,N}$ for any $\g\in\G_0(N)$, we obtain from \eqref{fgpnpt0} that 
\begin{align}\label{fgpnptnm}\nonumber
F_{\G^{(p)},n}(p\t)=&\sum_{G\in\Gamma_\infty\backslash\Gamma_0(N)}f_n(pG(\t))+\sum_{G\in\Gamma_\infty\backslash\Gamma_0(N)}f_n(G(\t))|W_{p,N}\\ 
&+\sum_{Q\in S^{(p)}}\left(\sum_{G\in\Gamma_\infty\backslash\Gamma_0(N)}f_n(pGW_{Q,N}(\t))+ \sum_{G\in\Gamma_\infty\backslash\Gamma_0(N)}f_n(GW_{Q,N}(\t))|W_{p,N}\right)
\end{align}
in a similar way to \eqref{fgpnptm}.

$\cdot$\ Case (2):  Again, $S=S^{(p)}$, we thus obtain from \eqref{fgpnptnm} that
$$F_{\G^{(p)},n}(p\t)=F_{\G,pn}(\t)+F_{\G,n}(\t)|W_{p,N},$$
which leads to the expansion formula.
For the other two cases, we need another decomposition of $\Gamma_\infty\backslash\Gamma_0(M)$ discovered in \eqref{dem}. Applying \eqref{dem} and \eqref{decg} into \eqref{fgnpt} with the fact $W_{p,N}^2=I$ (mod $\G_0(N)$), we have
\begin{align}\label{fgpnt} \nonumber
F_{\G^{(p)},n}(\t)=& \sum_{\G_\infty\gamma\in\Gamma_\infty\backslash\Gamma_0(N)}f_n(\G_\infty\alpha_\gamma\omega_M(\t))+\sum_{G\in\Gamma_\infty\backslash\Gamma_0(N)}f_n(GW_{p,N}(\t))|W_{p,N}\\ \nonumber
&+\sum_{Q\in S^{(p)}}\left(\sum_{\G_\infty\gamma\in\Gamma_\infty\backslash\Gamma_0(N)}f_n(\G_\infty\alpha_\gamma\omega_MW_{Q,M}(\t))+\sum_{G\in\Gamma_\infty\backslash\Gamma_0(N)}f_n(GW_{Q,M}W_{p,N}(\t))|W_{p,N}\right)\\ 
=& \sum_{G\in\Gamma_\infty\backslash\Gamma_0(N)}f_n(pGW_{p,N}(\t))+\sum_{G\in\Gamma_\infty\backslash\Gamma_0(N)}f_n(GW_{p,N}(\t))|W_{p,N}\\ \nonumber
&+\sum_{Q\in S^{(p)}}\left(\sum_{G\in\Gamma_\infty\backslash\Gamma_0(N)}f_n(pGW_{pQ,N}(\t))+\sum_{G\in\Gamma_\infty\backslash\Gamma_0(N)}f_n(GW_{pQ,N}(\t))|W_{p,N}\right).
\end{align}
The second equality above was obtained in a similar fashion as above and by means of the operator identity $\omega_M= V_pW_{p,N}$.

$\cdot$\ Case (1): Observe that
\begin{equation}\label{c1s}
S=S^{(p)}\cup \{p\}\cup  \{pQ\,|\,Q\in S^{(p)}\}.
\end{equation}
Thus by combining results in \eqref{fgpnptnm} and \eqref{fgpnt}, we conclude that 
\begin{equation*}
F_{\G^{(p)},n}(p\t)+F_{\G^{(p)},n}(\t)=F_{\G,np}(\t)+F_{\G,n}(\t)|W_{p,N}=F_{\G,np}(\t)+F_{\G,n}(\t),
\end{equation*}
which proves the expansion formula.

$\cdot$\ Case (3): Since $p\nmid M$ and $pQ_0\parallel N$, $Q_0\parallel N$ and $p\nmid Q_0$. Also since $W_{p,N}W_{Q_0,N}=W_{pQ_0,N}$, 
we have from \eqref{fgpnptnm} that 
\begin{align}\label{fgpnpw} \nonumber
F_{\G^{(p)},n}(p\t)|W_{Q_0,N}
=& \sum_{G\in\Gamma_\infty\backslash\Gamma_0(N)}f_n(pGW_{Q_0,N}(\t)) + \sum_{G\in\Gamma_\infty\backslash\Gamma_0(N)}f_n(G W_{pQ_0,N}(\t)) \\  
& +\sum_{Q\in S^{(p)}}\left(\sum_{G\in\Gamma_\infty\backslash\Gamma_0(N)}f_n(pGW_{Q,N}W_{Q_0,N}(\t)) + \sum_{G\in\Gamma_\infty\backslash\Gamma_0(N)}f_n(GW_{Q,M} W_{pQ_0,N}(\t))\right).
\end{align}
Moreover, \eqref{fgpnt} can be restated as follows:
\begin{align}\label{fgpnt2} \nonumber
F_{\G^{(p)},n}(\t)=& \sum_{G\in\Gamma_\infty\backslash\Gamma_0(N)}f_n(pGW_{pQ_0,N}W_{Q_0,N}(\t))+\sum_{G\in\Gamma_\infty\backslash\Gamma_0(N)}f_n(G(\t))\\ 
&+\sum_{Q\in S^{(p)}}\left(\sum_{G\in\Gamma_\infty\backslash\Gamma_0(N)}f_n(pGW_{Q,N}W_{pQ_0,N}W_{Q_0,N}(\t))+\sum_{G\in\Gamma_\infty\backslash\Gamma_0(N)}f_n(GW_{Q,N}(\t))\right).
\end{align}
Like in Case (4), we derive from \eqref{fgpnpw} and \eqref{fgpnt2} that
$$F_{\G^{(p)},n}(\t)+F_{\G^{(p)},n}(p\t)|W_{Q_0,N}=F_{\G,n}(\t)+F_{\G,pn}(\t)|W_{Q_0,N}.$$
This ends the proof of the expansion formulas.
\end{proof}%

Let us recall from Lemma 2.2 of \cite{Koike} or \cite{AL} that the following decomposition holds:
\begin{equation}\label{decomposition}
\Gamma_0(N)=\begin{cases} \displaystyle\bigcup_{j=0}^{p-1}\Gamma_0(N,p)T^j, & \hbox{ if } p|N,  \\
\displaystyle\bigcup_{j=0}^{p-1}\Gamma_0(N,p)T^j \bigcup \Gamma_0(N,p)\mathcal \omega_N, & \hbox{ if } p\nmid N. 
\end{cases}
\end{equation}
Also for $\gamma\in\Gamma_0(N)$ and $\nu_p=\sm 1&0\\0&p\esm$, one can easily obtain that
\begin{equation}\label{vp}\Gamma_\infty\gamma \nu_p= \nu_p \Gamma_\infty^{(p)}\alpha_\gamma
\end{equation}
and if $p\nmid M$, the Atkin-Lehner involution $W_{p,pM}$ can be expressed as
\begin{equation}\label{vpwp}
W_{p,pM}=\nu_p \omega_M.
\end{equation}
Moreover, we note that $V_p\nu_p=\sm p&0\\0&p\esm$, and so $V_p$ and $\nu_p$ act on functions as the inverse to each other.
Thus we refer Lemma \ref{ko} for the action of $\nu_p$.

Now we are ready to prove compression formulas for $j_{\G,n}$.
\begin{prop}[{\it{Compression formulas}}] \label{comp} Let $p$ be a prime divisor of $N$ and write $N=pM$.  Then the following hold:
\begin{itemize}
\item {\rm (Case 1)} $j_{\G,n}(\t)|{U^*_p}+j_{\G,n}(\t)=j_{\G^{(p)},n}(\t)+pj_{\G^{(p)},\frac{n}{p}}(\t)$.
\item {\rm (Case 2)} $j_{\G,n}(\t)|{U^*_p}+j_{\G,n}(\t)|W_{p,N}=pj_{\G^{(p)},\frac{n}{p}}(\t)$.
\item {\rm (Case 3)} $j_{\G,n}(\t)|{U^*_p}+j_{\G,n}(\t)|W_{p,N}=j_{\G^{(p)},n}(\t)|W_{Q_0,M}+pj_{\G^{(p)},\frac{n}{p}}(\t)$.
\item {\rm (Case 4)} $j_{\G,n}(\t)|{U^*_p}=j_{\G^{(p)},n}(\t)|W_{Q_0,M}+pj_{\G^{(p)},\frac{n}{p}}(\t)$.
\item {\rm (Case 5)} $j_{\G,n}(\t)|{U^*_p}=pj_{\G^{(p)},\frac{n}{p}}(\t)$.
\end{itemize}
\end{prop}
\begin{proof}
Like proofs for expansion formulas, we prove the results for $F_{\G,n}(\t,s)$ and use analytic continuation. 
We first note that the $U^*_p$ operator defined in  \eqref{ups} can be written as
\begin{equation}\label{up}
f|{U^*_p}=\sum_{j=0}^{p-1}f|{\nu_pT^{j}}.
\end{equation}
We first prove the three cases when $p\nmid M$. 

$\cdot$\ Case (1): Recall that $S=S^{(p)}\cup \{p\}\cup  \{pQ\,|\,Q\in S^{(p)}\}.$  Employing \eqref{up} in \eqref{fgnpt} then,  we obtain 
\begin{align}\label{fgnupt}\nonumber
F_{\G,n}(\t)|{U^*_p}&=\sum_{j=0}^{p-1}\left(  \sum_{G\in\Gamma_\infty\backslash\Gamma_0(N)}f_n(G \nu_p T^j(\t)) +\sum_{Q\in S^{(p)}} \sum_{G\in\Gamma_\infty\backslash\Gamma_0(N)}f_n(G W_{Q,N} \nu_p T^j(\t)) \right.\\
&\qquad\left.+\sum_{G\in\Gamma_\infty\backslash\Gamma_0(N)}f_n(G W_{p,N} \nu_p T^j(\t))
+\sum_{Q\in S^{(p)}} \sum_{G\in\Gamma_\infty\backslash\Gamma_0(N)}f_n(G W_{pQ,N} \nu_p T^j(\t))\right)\\
&\qquad=:(A)+(B)+(C)+(D), \nonumber
\end{align}
where $(A),(B),(C),(D)$ denote the four summands of \eqref{fgnupt} in order.

Since $p\nmid M$, using Lemma \ref{11} and (\ref{vp}) in the first, \eqref{decomposition} in the second, \eqref{vp} in the third, and \eqref{vpwp} in the last equality, we have 
\begin{align}\label{(A)} \nonumber
(A)&=\sum_{j=0}^{p-1}\left(\sum_{D\in\Gamma_\infty^{(p)}\backslash\Gamma_0(M,p)}f_n(\nu_p DT^j(\t))\right) \\  \nonumber
&=\sum_{D\in\Gamma_\infty^{(p)}\backslash\Gamma_0(M)}f_n(\nu_p D(\t)) - \sum_{D\in\Gamma_\infty^{(p)}\backslash\Gamma_0(M,p)}f_n(\nu_p D \omega_M(\t))\\
&=\sum_{D\in\Gamma_\infty^{(p)}\backslash\Gamma_0(M)}f_n(\nu_p D(\t)) - \sum_{G\in\Gamma_\infty\backslash\Gamma_0(N)}f_n(G\nu_p\omega_M(\t))\\
&=\sum_{D\in\Gamma_\infty^{(p)}\backslash\Gamma_0(M)}f_n(\nu_p D(\t)) - \sum_{G\in\Gamma_\infty\backslash\Gamma_0(N)}f_n(G(\t))|W_{p,N}.\nonumber
\end{align}

Using Lemma \ref{11}, (\ref{vp}), and \eqref{ko1} of Lemma \ref{ko} in the first equality and applying the exactly same arguments as above for the rest, we find that
\begin{align}\label{(B)}\nonumber
(B)&=\sum_{j=0}^{p-1}\left(\sum_{Q\in S^{(p)}} \sum_{j=0}^{p-1}  \sum_{D\in\Gamma_\infty^{(p)}\backslash\Gamma_0(M,p)} f_n(\nu_p DT^jW_{Q,M}(\t))\right)\\
&=\sum_{Q\in S^{(p)}}\sum_{D\in\Gamma_\infty^{(p)}\backslash\Gamma_0(M)}f_n(\nu_p D W_{Q,M}(\t)) - \sum_{Q\in S^{(p)}}\sum_{G\in\Gamma_\infty\backslash\Gamma_0(N)}f_n(GW_{Q,N}(\t))|W_{p,N}.
\end{align}

Now utilizing $\beta_\g V_p=\g W_{p,N}$ or $\beta_\g=\g W_{p,N} \nu_p$ and then \eqref{dem2}, we obtain
\begin{align}\label{(C)}\nonumber
(C)&=\sum_{j=0}^{p-1}\sum_{\G_\infty\gamma\in\Gamma_\infty\backslash\Gamma_0(N)}f_n(\G_\infty \beta_\gamma T^j(\t))\\
&=\sum_{D\in\Gamma_\infty\backslash\Gamma_0(M)}f_n(D (\t)) - \sum_{G\in\Gamma_\infty\backslash\Gamma_0(N)}f_n(G(\t))\\ \nonumber
&=\sum_{D\in\Gamma_\infty\backslash\Gamma_0(M)}f_n(D (\t)) - \sum_{G\in\Gamma_\infty\backslash\Gamma_0(N)}f_n(GW_{p,N}(\t))|W_{p,N}. 
\end{align}

In a similar way, we find that
\begin{align}\label{(D)}\nonumber
(D)&=\sum_{Q\in S^{(p)}} \sum_{j=0}^{p-1} \sum_{\G_\infty\gamma\in\Gamma_\infty\backslash\Gamma_0(N)}f_n(\G_\infty \beta_\gamma T^j W_{Q,M}(\t))\\
&=\sum_{Q\in S^{(p)}}\sum_{D\in\Gamma_\infty\backslash\Gamma_0(M)}f_n(D W_{Q,M}(\t)) - \sum_{Q\in S^{(p)}}\sum_{G\in\Gamma_\infty\backslash\Gamma_0(N)}f_n(GW_{pQ,N}(\t))|W_{p,N}.
\end{align}

On the other hand, using $\Gamma_\infty^{(p)}\backslash\Gamma_0(M)=\{D, TD, T^2D, \dots, T^{p-1}D|\, D\in \Gamma_\infty\backslash\Gamma_0(M)\}$ and \eqref{fn} in the first and second equality below, respectively, we obtain that
\begin{align*}
\sum_{D\in\Gamma_\infty^{(p)}\backslash\Gamma_0(M)}f_n(\nu_p D(\t)) &=2\pi\sum_{j=0}^{p-1}\sum_{D\in\Gamma_\infty\backslash\Gamma_0(M)}f_n(\nu_p T^j D(\t))\\
&=\sum_{j=0}^{p-1}\sum_{D\in\Gamma_\infty\backslash\Gamma_0(M)}e\left(-\frac{jn}{p}\right)f_n(\nu_p D(\t))\\
&=\left\{ \sum_{j=0}^{p-1} e\left(-\frac{jn}{p}\right) \right\} \left\{\sum_{D\in\Gamma_\infty\backslash\Gamma_0(M)} f_n(\nu_p D(\t))\right\}.
\end{align*}

By the same reason, we have
\begin{align*}
\sum_{Q\in S^{(p)}}\sum_{D\in\Gamma_\infty^{(p)}\backslash\Gamma_0(M)}f_n(\nu_p D W_{Q,M}(\t))
=\sum_{Q\in S^{(p)}}\left\{ \sum_{j=0}^{p-1} e\left(-\frac{jn}{p}\right) \right\} \left\{\sum_{D\in\Gamma_\infty\backslash\Gamma_0(M)} f_n(\nu_p DW_{Q,M}(\t))\right\}.
\end{align*}
It thus follows from \eqref{fns} and
$$\sum_{j=0}^{p-1} e\left(-\frac{jn}{p}\right) = \begin{cases} p, & \hbox{ if } p|n\\
0, & \hbox{ if } p\nmid n \end{cases}$$
that 
\begin{equation}\label{twosum}
\sum_{D\in\Gamma_\infty^{(p)}\backslash\Gamma_0(M)}f_n(\nu_p D(\t))+\sum_{Q\in S^{(p)}}\sum_{D\in\Gamma_\infty^{(p)}\backslash\Gamma_0(M)}f_n(\nu_p D W_{Q,M}(\t))=pF_{\G^{(p)},\frac{n}{p}}(\t).
\end{equation}
Finally, from \eqref{fgnupt}, \eqref{(A)}, \eqref{(B)}, \eqref{(C)}, \eqref{(D)} and \eqref{twosum}, we establish
$$F_{\G,n}(\t)|{U^*_p}(\t)+F_{\G,n}(\t)|W_{p,N}=F_{\G^{(p)},n}(\t)+pF_{\G^{(p)},\frac{n}{p}}(\t).$$
Since $p\in S$, we have
$$F_{\G,n}(\t)|{U^*_p}+F_{\G,n}(\t)=F_{\G^{(p)},n}(\t)+pF_{\G^{(p)},\frac{n}{p}}(\t),$$
which proves the compression formula.

$\cdot$\ Case (2): Since $S=S^{(p)}$, only $(A)$ and $(B)$ in \eqref{fgnupt} can occur.  Hence we have 
$$F_{\G,n}(\t)|{U^*_p}+F_{\G,n}(\t)|W_{p,N}=pF_{\G^{(p)},\frac{n}{p}}(\t)$$
and prove the compression formula.

$\cdot$\ Case (3): We may write
\begin{equation}\label{coca3} F_{\G,n}(\t)|{U^*_p}=(A)+(B)+(C')+(D'),\end{equation}
 where $A$ and $B$ are the same as in \eqref{fgnupt} and $(C')$ and $(D')$ are given below:
 \begin{align}\label{(C')}
(C')&:=\sum_{j=0}^{p-1}\sum_{G\in\Gamma_\infty\backslash\Gamma_0(N)}f_n(G W_{pQ_0,N} \nu_p T^j(\t))
\end{align}
and 
\begin{align}\label{(D')}
(D')&:=\sum_{Q\in S^{(p)}}\sum_{j=0}^{p-1}\sum_{G\in\Gamma_\infty\backslash\Gamma_0(N)}f_n(G W_{pQ_0,N} W_{Q,N} \nu_p T^j(\t)).
\end{align}
Using \eqref{ko1} of Lemma \ref{ko} and $\beta_\g=\g W_{p,N} \nu_p$ in the first equality and \eqref{dem2}  in the second equality below, we find that
\begin{align}\label{(C'-1)}\nonumber
(C')
&=\sum_{j=0}^{p-1}\sum_{\G_\infty\gamma\in\Gamma_\infty\backslash\Gamma_0(N)}f_n(\G_\infty\beta_\gamma T^jW_{Q_0,M}(\t))\\
&=\sum_{D\in\Gamma_\infty\backslash\Gamma_0(M)}f_n(D (\t))|W_{Q_0,M} - \sum_{G\in\Gamma_\infty\backslash\Gamma_0(N)}f_n(G W_{pQ_0,N}(\t))|W_{p,N}.
\end{align}

Similarly, we obtain that

\begin{align}\label{(D'-1)}\nonumber
(D')
&=\sum_{Q\in S^{(p)}}\sum_{j=0}^{p-1}\sum_{\G_\infty\gamma\in\Gamma_\infty\backslash\Gamma_0(N)}f_n(\G_\infty\beta_\gamma T^j W_{Q,M} W_{Q_0,M} (\t))\\
&=\sum_{Q\in S^{(p)}}\sum_{D\in\Gamma_\infty\backslash\Gamma_0(M)}f_n(D W_{Q,M} (\t))|W_{Q_0,M} - \sum_{Q\in S^{(p)}}\sum_{G\in\Gamma_\infty\backslash\Gamma_0(N)}f_n(G W_{pQ_0,N} W_{Q,N}(\t))|W_{p,N}.
\end{align}
Thus by \eqref{coca3}, \eqref{(A)}, \eqref{(B)}, \eqref{(C'-1)}, \eqref{(D'-1)} and \eqref{twosum}, we arrive at 
$$F_{\G,n}(\t)|{U^*_p}+F_{\G,n}(\t)|W_{p,N}=F_{\G^{(p)},n}(\t)|W_{Q_0,M}+pF_{\G^{(p)},\frac{n}{p}}(\t),$$
which proves the compression formula.

It remains to prove the two last cases when $p\mid M$.

$\cdot$\ Case (5): Like in Case (2), $S=S^{(p)}$, and hence we have only $(A)$ and $(B)$ in \eqref{fgnupt}. Furthermore, by \eqref{decomposition}, we have only the first summands in the right-hand side of the second equalities of \eqref{(A)} and \eqref{(B)}. Thus from \eqref{twosum}, we obtain that
$$F_{\G,n}|{U_p}(\t)=pF_{\G^{(p)},\frac{n}{p}}(\t),$$
from which we have the compression formula.

$\cdot$\ Case (4): As in Case (3), we can write
\begin{equation}\label{coca4} F_{\G,n}(\t)|{U^*_p}=(A)+(B)+(C')+(D'),\end{equation}
 where $(A)$ and $(B)$ are the same as in \eqref{fgnupt} and $(C')$ and $(D')$ are defined in \eqref{(C')} and \eqref{(D')}, respectively.
Using the first identity in \eqref{ko4} of Lemma \ref{ko}, \eqref{ko3} of Lemma \ref{ko}, and \eqref{dem3} in turn, we find that
\begin{align}\label{(C'')}\nonumber
(C')
&=\sum_{j=0}^{p-1}\sum_{G\in\Gamma_\infty\backslash\Gamma_0(N)}f_n(GW_{Q_0,M}T^j(\t))\\
&=\sum_{j=0}^{p-1}\sum_{G\in\Gamma_\infty\backslash\Gamma_0(N)}f_n(G\delta_jW_{Q_0,M}(\t))=\sum_{D\in\Gamma_\infty\backslash\Gamma_0(M)}f_n(D (\t))|W_{Q_0,M}.
\end{align}
Similarly, we obtain that
\begin{align}\label{(D'')}\nonumber
(D')
&=\sum_{Q\in S^{(p)}}\sum_{j=0}^{p-1}\sum_{G\in\Gamma_\infty\backslash\Gamma_0(N)}f_n(G W_{Q,M} W_{Q_0,M} T^j (\t))\\
&=\sum_{Q\in S^{(p)}}\sum_{j=0}^{p-1}\sum_{G\in\Gamma_\infty\backslash\Gamma_0(N)}f_n(G \delta_j W_{Q,M} W_{Q_0,M} (\t))\\ \nonumber
&=\sum_{Q\in S^{(p)}}\sum_{D\in\Gamma_\infty\backslash\Gamma_0(M)}f_n(DW_{Q,M}(\t))|W_{Q_0,M}.
\end{align}
%
%
Employing \eqref{decomposition} in the proof of \eqref{(A)} and \eqref{(B)}, then applying the results with \eqref{(C'')}, \eqref{(D'')} and \eqref{twosum} into \eqref{coca4},  we have
$$F_{\G,n}(\t)|{U^*_p}=F_{\G^{(p)},n}(\t)|W_{Q_0,M}+pF_{\G^{(p)},\frac{n}{p}}(\t),$$
from which we prove the last compression formula.

\end{proof}
We obtain \eqref{ep-pli} for $j_{\G,n}$ from Propositions \ref{TnN} and \ref{comp} if $p\mid N$. If $p\nmid N$, the left-hand side of \eqref{ep-pli} gives $J_{\G,n}|T_p$. Hence \eqref{ep-pli} follows from the same arguments as in \cite[Lemma 5.1 (5.2)]{KK} or \cite[Theorem 1.1 (2)]{BKLOR} when $p\nmid N$.  Therefore, Theorem \ref{p-pli} is proved.

\section{Replication formula and generalized Hecke operator on $J_{\G,n}(\t)$}
Now we generalize the {\it $p$-plication formula}.
\begin{thm}[{\it{$p^k$-plication formula}}] \label{Hecke} For any prime $p$ and positive integer $k$, we have
\begin{equation}\label{Hecke1}
J_{\G,n}|T(p^k)(\t)=\sum_{p^j|(n,p^k)} p^j J_{\G,\frac{p^kn}{p^{2j}}}^{(p^j)}(\t).
\end{equation}
\end{thm}
\begin{proof}
We use induction on $k$. When $k=1$, \eqref{Hecke1} equals \eqref{ep-pli}, the {\it $p$-plication formula}.
We assume that \eqref{Hecke1} is true for $k$ and compute the equation (\ref{k1}) below in two ways:
\begin{equation}\label{k1}\left(J_{\G,n}|T(p^k)\right)|{U^*_p}(\t)+\left(J_{\G,n}^{(p)}|T(p^k)\right)(p\t).
\end{equation}
First by the definition of Hecke operator \eqref{deft}, \eqref{k1} equals
\begin{equation}\left(\sum_{i=0}^{k} J_{\G,n}^{(p^i)}|{U^*_{p^{k-i}}}(p^i\t)\right)|{U^*_p} + \sum_{i=0}^{k} J_{\G,n}^{(p^{i+1})}|{U^*_{p^{k-i}}}(p^{i+1}\t).
\end{equation}
Applying the fact $V_pU^*_p=p$ here, we find that this equals
\begin{align}
&J_{\G,n}|{U^*_{p^{k+1}}}(\t)+p\sum_{i=1}^{k} J_{\G,n}^{(p^i)}|{U^*_{p^{k-i}}}(p^{i-1}\t)+\sum_{i=0}^{k} J_{\G,n}^{(p^{i+1})}|{U^*_{p^{k-i}}}(p^{i+1}\t)\\ \notag
&=J_{\G,n}|{U^*_{p^{k+1}}}(\t)+\sum_{i=1}^{k+1} J_{\G,n}^{(p^{i})}|{U^*_{p^{k+1-i}}}(p^{i}\t)+p\sum_{i=0}^{k-1} (J_{\G,n}^{(p)})^{(p^i)}|{U^*_{p^{k-1-i}}}(p^{i}\t).
\end{align}
By \eqref{deft}, we eventually have
\begin{equation}\label{k11}\left(J_{\G,n}|T(p^k)\right)|{U^*_p}(\t)+\left(J_{\G,n}^{(p)}|T(p^k)\right)(p\t)=J_{\G,n}|T(p^{k+1})(\t)+pJ_{\G,n}^{(p)}|T(p^{k-1})(\t).\end{equation}

On the other hand, by the induction hypothesis, \eqref{k1} equals
\begin{equation}\label{operator2}
\sum_{p^j|(n,p^k)} p^j J_{\G,\frac{p^kn}{p^{2j}}}^{(p^j)}|{U^*_p}(\t) + \sum_{p^j|(n,p^k)} p^{j} J_{\G,\frac{p^kn}{p^{2j}}}^{(p^{j+1})}(p\t)
=\sum_{p^j|(n,p^k)} p^j\left( J_{\G,\frac{p^kn}{p^{2j}}}^{(p^j)}|{U^*_p}(\t) + J_{\G,\frac{p^kn}{p^{2j}}}^{(p^{j})(p)}(p\t)\right).
\end{equation}
Using {\it{$p$-plication formula}} \eqref{ep-pli}, we find that the right hand side of the equation \eqref{operator2} equals
\begin{equation}\label{k12}
\sum_{p^j|(n,p^k)} p^j\left( J_{\G,\frac{p^{k+1}n}{p^{2j}}}^{(p^j)}(\t) + pJ_{\G,\frac{p^kn}{p^{2j+1}}}^{(p^{j+1})}(\t)\right).
\end{equation}
Since this equals 
$$\sum_{p^j|(n,p^{k+1})} p^j J_{\G,\frac{p^{k+1}n}{p^{2j}}}^{(p^j)}(\t) - p^{k+1} J_{\G,\frac{p^{k+1}n}{p^{2k+2}}}^{(p^{k+1})}(\t) + \displaystyle\sum_{p^j|(n,p^{k-1})} p^{j+1}J_{\G,\frac{p^{k-1}n}{p^{2j}}}^{(p^{j+1})}(\t) +  p^{k+1} J_{\G,\frac{p^kn}{p^{2k+1}}}^{(p^{k+1})}(\t)$$
if  $p^{k+1}|n$ and equals
$$\sum_{p^j|(n,p^{k+1})} p^j J_{\G,\frac{p^{k+1}n}{p^{2j}}}^{(p^j)}(\t) + \displaystyle\sum_{p^j|(n,p^{k-1})}p^{j+1}J_{\G,\frac{p^{k-1}n}{p^{2j}}}^{(p^{j+1})}(\t)$$
otherwise,
we see that \eqref{k12} equals
$$\sum_{p^j|(n,p^{k+1})} p^j J_{\G,\frac{p^{k+1}n}{p^{2j}}}^{(p^j)}(\t) + p\sum_{p^j|(n,p^{k-1})}p^{j}J_{\G,\frac{p^{k-1}n}{p^{2j}}}^{(p)(p^{j})}(\t).$$
Hence by the induction hypothesis again, we finally have
\begin{equation}\label{k13}
\left(J_{\G,n}|T(p^k)\right)|{U^*_p}(\t)+\left(J_{\G,n}^{(p)}|T(p^k)\right)(p\t)=\sum_{p^j|(n,p^{k+1})} p^j J_{\G,\frac{p^{k+1}n}{p^{2j}}}^{(p^j)}(z) + p J_{\G,n}^{(p)}|T(p^{k-1})(\t).
\end{equation}
Comparing \eqref{k11} with \eqref{k13}, we obtain \eqref{Hecke1} for $k+1$. This completes the proof.
\end{proof}

For a prime factorization $m=p_1p_2\cdots p_r$ of any given positive integer $m$, the $m$-plicate of $J_{\G,n}$ is given by \eqref{defp} as 
$$J_{\G,n}^{(m)}=J_{\G,n}^{(p_1)(p_2)\cdots(p_r)}.$$
One can easily check that $J_{\G,n}^{(m)}$ is independent of a prime factorization of $m$.

\begin{prop}\label{Hecke2} We have, for any distinct primes $p$ and $\ell$ and positive integers $k$ and $r$,
\begin{enumerate}
\item[(1)] $T(p^k)\circ T(\ell^r)=T(\ell^r)\circ T(p^k)$
\item[(2)] $T(p^k)\circ T(p)=T(p^{k+1}) + pI_p\circ T(p^{k-1})$, where $J_{\G,n}|I_p = J_{\G,n}^{(p)}$.
\item[(3)] $I_p\circ I_\ell = I_\ell \circ I_p$.
\end{enumerate}
\end{prop}
\begin{proof} (2) is already proved in \eqref{k11}. (3) is also true because the definition of {\it{m-plication}} of $J_{\G,n}$ is independent of any prime decompositions of $m$. For (1), we use Theorem \ref{Hecke} so that we obtain
\begin{align*}
J_{\G,n}|T(p^k)|T(\ell^r)(\t)&=\sum_{p^j|(n,p^k)} p^jJ_{\G,\frac{p^{k}n}{p^{2j}}}^{(p^j)}|T(\ell^r)(\t)\\
&=\sum_{p^j|(n,p^k)} p^j \sum_{\ell^i|(p^{k-2j}n,\ell^r)}\ell^iJ_{\G,\frac{p^{k}\ell^{r}n}{p^{2j}\ell^{2i}}}^{(p^j)(\ell^i)}(\t)\\
&=\sum_{p^j\ell^i|(n,p^k\ell^r)} p^j\ell^i J_{\G,\frac{p^{k}\ell^{r}n}{p^{2j}\ell^{2i}}}^{(p^j)(\ell^i)}(\t).
\end{align*}
Since $J_{\G,n}^{(p^j)(\ell^i)}=J_{\G,n}^{(p^j\ell^i)}=J_{\G,n}^{(\ell^i)(p^j)}$, this is equal to $J_{\G,n}|T(\ell^r)|T(p^k)(\t)$, which proves (1).
\end{proof}
By Proposition \ref{Hecke2} (1), we may define the Hecke operator $T(m)$ for any positive integer $m$ by

\begin{equation}\label{tm}T(m):=T(p^{k_1})\circ T(p^{k_2})\circ \cdots T(p^{k_r}),\end{equation}
where $m=p^{k_1}p^{k_2}\cdots p^{k_r}$ and $p_1,p_2,\dots,p_r$ are distinct primes.

Now we are ready to prove the {\it{replication formula}} for $J_{\G,n}$.
\begin{proof}[Proof of Theorem \ref{replic}] We use induction on the number of prime divisors of $m$.
First, let $m=p^k$. Then \eqref{Hecke3-1} follows from Theorem \ref{Hecke} and \eqref{Hecke3-2} follows from the definition of $T(p^k)$.
 Assume that both \eqref{Hecke3-1} and \eqref{Hecke3-2} hold for any integer that has less prime divisors than $m$ does. Now let $p^k\| m$, then $T(m)=T(\frac{m}{p^k})\circ T(p^k)$.
By the induction hypothesis, we have
\begin{align}\label{tm1}
J_{\G,n}|T(m)(\t)&=\left(\sum_{d|(n,\frac{m}{p^k})}dJ_{\G,\frac{mn}{p^kd^2}}^{(d)}\right)|T(p^k)(\t)\cr
&=\sum_{d|(n,\frac{m}{p^k})}\sum_{d'|(\frac{mn}{p^kd^2},p^k)}dd'J_{\G,\frac{mn}{(dd')^2}}^{(d)(d')}(\t).
\end{align}
Since $(\frac{m}{p^k},p^{k})=1$, we have $(\frac{mn}{p^kd^2},p^k)=(n,p^k)$. Hence the far right hand side of \eqref{tm1} equals
$$\sum_{d|(n,m)}dJ_{\G,\frac{mn}{d^2}}^{(d)}(\t).$$
Thus we have \eqref{Hecke3-1}.

By the induction hypothesis again, we have
\begin{align}\label{tm2}
J_{\G,n}|T(m)(\t)&=\left(\sum_{d|\frac{m}{p^k}}J_{\G,n}^{(d)}|U^*_{\frac{m}{p^{k}d}}(dz)\right)|T(p^k)\cr
&=\sum_{d|\frac{m}{p^k}}\sum_{p^i|p^k}J_{\G,n}^{(dp^i)}|U^*_{\frac{m}{p^{k}d}}|V_d|U^*_{p^{k-i}}|V_{p^i}(\t).
\end{align}
As we have $V_dU^*_{p^{k-i}}=U^*_{p^{k-i}}V_d$ because of $(d,p^{k-i})=1$, we find that the far right hand side of \eqref{tm2} equals
$$\sum_{d|\frac{m}{p^k},p^i|p^k}J_{\G,n}^{(dp^i)}|U^*_{\frac{m}{p^{i}d}}(dp^i\t),$$
which proves \eqref{Hecke3-2}.
\end{proof}

\section{Duality between Modular forms on $X(\G)$ of higher genus}
Suppose the genus $g\geq 1$ and $i\infty$ is not a Weierstrass point on $X(\G)$. By Weierstrass gap theorem, there are two generators $X_\G$ and $Y_\G$ of $\C(X(\G))$ which have order $g+1$ and $g+2$ at $i\infty$, respectively.
Yang \cite[\S 4.1]{Yang} found generators of the function field $\mathbb{C}(X_0(N))$ for many values of $N$.
For example, two generators of $\C(X_0(11))$ on the genus 1 curve $X_0(11)$ are given by 
$$X_{11}= q^{-2}+2q^{-1}+4+5q+8q^2+q^3+7q^4-11q^5+\cdots$$ and 
$$Y_{11}= q^{-3}+3q^{-2}+7q^{-1}+12+17q+26q^2+19q^3+37q^4-15q^5+\cdots.$$
Moreover, two generators of the function field $\C(X_0(22))$ on the genus 2 curve $X_0(22)$ are given by 
$$X_{22}= q^{-3}+q^{-1}+4+2q+2q^2+2q^5+\cdots$$ and 
$$Y_{22}= q^{-4}-q^{-3}+2q^{-2}+q^{-1}+3-2q+q^2+\cdots.$$

By adopting his method, one can find two generators of the function field $\mathbb{C}(X(\G))$ with rational coefficients and construct the reduced row echelon basis $\{f_{\G,m}(\t):m=0\ \mathrm{or}\ m>g\}$ of $M_0^{!,\i}(\G)$ with  them. 
As described in Introduction, using $\{f_{\G,m}\}$ and the basis given in \eqref{mgf} of $S_2(\G)$, one can construct a basis $\{h_{\G,n}(\t):n\geq-g\}$ of $S_2^{!,\infty}(\G)$ in the form of \eqref{gF}. We prove that  $\{f_{\G,m}(\t)\}$ and $\{h_{\G,n}(\t)\}$ satisfy the Zagier duality.
\begin{proof}[Proof of Theorem \ref{Tgrid}]
For each cusp $s$, we take $\gamma=\left(\begin{smallmatrix}
  a&b\\c &d
  \end{smallmatrix}\right) \in SL_{2}(\Bbb{Z})$ such that $\gamma i\infty =
  s$.
Then there exists a unique positive real number $h_{s}$, so called {\it width} of the cusp $s$ such that
$$
\gamma^{-1}\G_{s}\gamma=\{\pm
\left(\begin{smallmatrix}
  1&h_{s}\\0 &1
  \end{smallmatrix}\right)^{m} | m\in \Bbb{Z}\},
$$
where $\G_{\t}$ is the stabilizer of $\t$ in $\G$.
For example, we have that $h_{i\infty}=1$ and $h_0=N$. 
Let $G=f_{\G,m}\cdot h_{\G,n}$ for any $m\geq g+1$ and $n\geq -g$.
The meromorphic weight two
modular form $G(\t)$ for
 $\G$ has a Fourier
expansion at each cusp $s$ in the form
$$
(G|_2\gamma)(\t)=\sum_{n\geq n_{0}^{(s)}}a_{n}^{(s)}q_{h_{s}}^{n}.
$$
Here $q_\ell=e^{2\pi i \t/\ell}$ and $|_{k}$ is the usual weight $k$ slash operator. Then for the corresponding differential
 $\omega_{G}=G(\t)d\t$ on $X(\G)$,
 using the canonical quotient map $\pi : \H\cup\Q\cup\{i\i\} \rightarrow X(\G)$,
  the Residue Theorem on compact Riemann surfaces states that
 $$\sum_{p\in X(\G)}\text{Res}_p \omega_{G} =0.$$
Let $1/e_{\tau}$ be the cardinality of $
\G_{\tau}/\{\pm1\}$ for each $\tau \in \H$. 
 Then we see that for each cusp $s$ and $\tau \in \H$,
$$
\text{Res}_{\pi(s)}\omega_{G}=\frac{h_{s}}{2 \pi i }a_{0}^{(s)}
\quad  \mathrm{and}\quad \hbox{Res}_{\pi(\tau)}\omega_{G}=e_{\tau}\hbox{Res}_{\tau}G(\t)
.
$$
(See \cite{Choi} for the computation of residues in detail.)  Since $G(\t)$ is holomorphic in $\H$,
 it follows from the Residue Theorem that 
\begin{equation} \label{RT}
\sum_{s}h_{s} a_0^{(s)}  =0.
\end{equation}
Obviously,  $a_0^{(i\infty)}=b_{\G}(n,m)+a_{\G}(m,n)$ and $a_0^{(s)}=0$ for each cusp $s$ other than $i\infty$.
Now the assertion immediately follows from \eqref{RT}.
\end{proof}

\section{Hecke action on $M_0^{!,\i}(\G)$}
Now we are ready to investigate the action of the generalized Hecke operator on $f_{\G,m}$. 
Theorem \ref{tpm} is an immediate consequence of Theorem \ref{fpm-hecke}. %

\begin{thm}\label{fpm-hecke}
Let $\G=N+S$ and $p\nmid N$ be a prime.  Assume that $i\infty$ is not a Weierstrass point on $X(\G)$. Then for any positive integers $r$ and $m\geq g+1$, we have
\begin{equation}\label{fpm-hecke1}
f_{\G,m}|T(p^r)(\t)=\sum_{p^j|(m,p^r)} p^j f_{\G,\frac{p^rm}{p^{2j}}}(\t)+\sum_{l=1}^g \sum_{p^j|(l,p^r)} p^j a_{\G}(m,-l) f_{\G,\frac{l p^r}{p^{2j}}}(\t).
\end{equation}
\end{thm}

\begin{proof} Recall from \eqref{fjg} that 
\begin{equation*}f_{\G,m}=J_{\G,m}+\sum_{l=1}^{g} a_{\G}(m,-l)J_{\G,l}.\end{equation*}
 Hence by \eqref{Hecke3-1} and \eqref{heckedef}, we have
  \begin{align}\label{a2}\nonumber
f_{\G,m}|T(p^r)(\t)&=\sum_{p^j|(m,p^r)} p^j J_{\G,\frac{p^rm}{p^{2j}}}^{(p^j)}(\t)+\sum_{l=1}^{g} a_{\G}(m,-l)\sum_{p^j|(l,p^r)}p^jJ_{\G,\frac{p^rl}{p^{2j}}}^{(p^j)}(\t)\\ 
&=\sum_{p^j|(m,p^r)} p^j f_{\G,\frac{p^rm}{p^{2j}}}(\t)+\sum_{l=1}^{g}\sum_{p^j|(l,p^r)}p^j a_{\G}(m,-l)f_{\G,\frac{p^rl}{p^{2j}}}(\t)\\ \nonumber
&\quad-\sum_{i=1}^{g}\left(\sum_{p^j|(m,p^r)} p^j a_{\G}(\frac{p^rm}{p^{2j}},-i)+\sum_{l=1}^{g}\sum_{p^j|(l,p^r)}p^j a_{\G}(m,-l)a_{\G}(\frac{p^rl}{p^{2j}},-i)\right)J_{\G,i}(\t).
\end{align}
For each $i$, if we let
\begin{align*}
A_i&=\sum_{p^j|(m,p^r)} p^j a_{\G}(\frac{p^rm}{p^{2j}},-i)+\sum_{l=1}^{g}\sum_{p^j|(l,p^r)}p^j a_{\G}(m,-l)a_{\G}(\frac{p^rl}{p^{2j}},-i),
\end{align*}
then proof ends if we show that $A_i=0$ for all $i=1,\cdots,g$.

First note that by Theorem \ref{Tgrid}, 
\begin{align}\label{ai}
A_i&=-\sum_{p^j|(m,p^r)} p^j b_{\G}(-i,\frac{p^rm}{p^{2j}})+\sum_{l=1}^{g}\sum_{p^j|(l,p^r)}p^j b_{\G}(-l,m)b_{\G}(-i,\frac{p^rl}{p^{2j}}).
\end{align}
Applying the classical Hecke operator $T_p$ on the cusp form $h_{\G,-i}$ in \eqref{mgf} for each $i=1,\cdots,g$,
we obtain by \cite[Proposition 5.3.1]{DS} that
\begin{equation}\label{gnst1}
h_{\G,-i}|T_{p^r}(\t)=\sum_{m=1}^\infty \left(\sum_{p^j|(m,p^r)} p^j b_{\G}(-i,{\frac{p^rm}{p^{2j}}})\right)q^m.
\end{equation}
On the other hand, since $h_{\G,-i}|T_{p^r}(\t)$ is a linear combination of $h_{\G,-l}(\t)$ ($1\leq l\leq g$),  we have
\begin{equation}\label{gnst2}
h_{\G,-i}|T_{p^r}(\t)=\sum_{l=1}^g \left(\sum_{p^j|(l,p^r)} p^j b_{\G}(-i,{\frac{p^rl}{p^{2j}}})\right) h_{\G,l}(\t).
\end{equation}
Comparing the coefficients of $q^m$ in \eqref{gnst1} and \eqref{gnst2}, we arrive at $A_i=0$.
\end{proof} 

\begin{rmk} 
(1) For simplicity, we proved Theorem \ref{fpm-hecke} and thus Theorem \ref{tpm} when $i\infty$ is not a Weierstrass point on $X(\G)$. But a similar congruence holds in general.

(2) When $p>g$ and $p\mid N$, we obtain a formula analogous to \eqref{fpm-hecke1} if $p\nmid m$:
\begin{equation}\label{fpm-hecke2}
f_{\G,m}|T(p^r)(\t)=f_{\G,p^rm}(\t)+\sum_{l=1}^g a_{\G}(m,-l)f_{\G,lp^r}(\t)
\end{equation}
by applying $T_{p^r}$ on $h_{\G,-i}(\t)$ again.  Note that this time, $T_p=U_p$ and $A_i$ in \eqref{ai} equals
$$A_i=-b_{\G}(-i,p^rm)+\sum_{l=1}^{g}b_{\G}(-l,m)b_{\G}(-i,lp^r).$$
Thus $f_{\G,m}|T(p^r)(\t)$ is still in $M_0^{!,\i}(\G).$ 

(3) When $p\mid N$ and $p\mid m$,  $T(p^r)$ may or may not preserve the space $M_0^{!,\infty}(\G)$.
For example, $f_{11,11}|T(11)(\t)$ is weakly holomorphic, while $f_{22+2,2}|T(2)(\t)$ is not. The reasons are as follows: Firstly, $f_{11,11}(\t)=q^{-11}-q^{-1}+O(q)\in M_0^{!}(11)$. By \eqref{fjg}, $f_{11,11}(\t)=J_{11,11}(\t)-J_{11,1}(\t)$. Hence by \eqref{heckedef}, \eqref{Hecke3-1}, and \eqref{hjN}, $f_{11,11}|T(11)(\t)=J_{11,121}(\t)-J_{11,11}(\t)+11J_{1,1}(\t).$  But since $f_{11,121}=J_{11,121}-J_{11,1}$, we have
$$f_{11,11}|T(11)(\t)=f_{11,121}(\t)-f_{11,11}(\t)+11J_{1,1}(\t)\in M_0^{!}(11).$$

Nextly, since $f_{22+2,2}=q^{-2}+q+2q^2+q^3+3q^4-3q^5+\cdots\in M_0^{!}(22+2)$, $f_{22+2,2}=J_{22+2,2}$. Regarding $\G^{(2)}=\G_0(11)$ and $J_{22+2,2}^{(2)}=J_{11,2}$, we obtain from \eqref{Hecke3-2} that
$$f_{22+2,2}|T(2)(\t)=J_{22+2,2}|U_2(\t)+J_{11,2}(2\t)=f_{22+2,2}|U_2(\t)+J_{11,2}|V_2(\t),$$
which is non-holomorphic.
\end{rmk}

Under certain conditions though, the Hecke operator is assured to preserve the space as given in Theorem \ref{heckepres}.

\begin{proof}[Proof of Theorem \ref{heckepres}] By \eqref{tm} and Theorem \ref{tpm}, it suffices to prove that $f|T(p^r)(\t)$ is weakly holomorphic for a prime $p\mid N$ and a positive integer $r$.
 Employing \eqref{deft}, \eqref{heckedef} and \eqref{def-fn}, we find that
\begin{equation}\label{ftpk}
f|T(p^r)(\t)=\sum_{i=0}^{r} f^{(p^i)}|{U^*_{p^{r-i}}}(p^i\t).
\end{equation}
By hypothesis,  $f^{(p^i)}\in M_0^{!,\infty}(\G^{(p^i)})$ for all $i=1,2,\dots, r$. Since $p\mid N$, for each $i$, $f^{(p^i)}|{U^*_{p^{r-i}}(\t)}\in M_0^{!}(\G^{(p^i)})$ by \cite[Lemma 6]{AL}, and hence $f^{(p^i)}|{U^*_{p^{r-i}}}(p^i\t)$ is weakly holomorphic.
Therefore, the result follows.
\end{proof}

\section{Replicability for higher genera}

In \cite{S}, Smith proposed another way to define a replicate of $f\in M_0^{!,\i}(\G)$ and also define a Hecke operator on $M_0^{!,\i}(\G)$ for non-zero genus. We will briefly describe his interesting construction.

Let $\A$ be an algebra of meromorphic functions in a disk $D$ around $0$, which have poles only at $0$.
For example, one can take $\A=M_0^{!,\infty}(\G)$. For a function $f$ meromorphic on $D$ with its only poles at $0$ (that is, $f$ has a $q$-expansion), define a projection $P_{\A}$ which maps $f$ to its closest function $f^*\in\A$ with respect to a metric on functions defined in the punctured disk $D-\{0\}$.
If there exists $f^*\in\A$ that has the same principal part with $f$, then $P_{\A}(f)=f^*$. 
Now define a Hecke operator by
\begin{equation}\label{gh1}
f|\tilde{T}_n=P_{\A}(f|T_{n}),
\end{equation}
where $T_n$ is the classical Hecke operator as before. If
\begin{equation}\label{gh2}
f|\tilde{T}_n(q)=\frac{1}{n}\sum_{ab=n} f^{[a]}|U_b(q^a)
\end{equation}
for certain functions defined by $q$-series $f^{[a]}$, then $f^{[a]}$ is called the $a$-plicate of $f$. (Smith also used the notation $f^{(a)}$ for the $a$-plicate of $f$. Here we use the square bracket to distinguish it with ours.) The functions on genus zero modular curves (with the principal part $q^{-1}$) satisfying \eqref{gh2} for the classical Hecke operator is called replicable.  
If replicates of $f\in\A$ exist, $f$ and its $p$-replicates satisfy the same congruence that was used to test the replicability of functions on genus zero modular curves, as pointed out in \cite{ACMS}.  
\begin{thm}\cite{S}\label{fp-con} Let $p$ be a prime. Suppose $f\in\A$ with rational integer coefficients and there exists a $p$-plicate $f^{[p]}$, then
$$f\equiv f^{[p]}\pmod p.$$
\end{thm}
\begin{proof} Since $pf|T_p=f(q^p)+pf|U_p(q)\equiv f(q^p)\equiv f(q)^p\pmod{p}$ and $pf|T_p$ and $f(q)^p$ have the same principal part modulo $p$, we have
$$pf|\tilde{T}_p\equiv f(q)^p\pmod{p}.$$
On the other hand, $pf|\tilde{T}_p=f^{[p]}(q^p)+pf|U_p(q)\equiv f^{[p]}(q^p)\pmod{p}$, and hence the result follows.
\end{proof}

Smith extended his notions so that the identity in \eqref{gh2} works with every element of $\A$. Set a new algebra  $\A^{[a]}$ to which $f^{[a]}$ belongs and extend the definition of the replicate by setting  $f^{[a]}=P_{\A^{[a]}}(f)$.  Then define a new Hecke operator $\hat{T}_n(q)$ by
\begin{equation}\label{gh3}
f|\hat{T}_n(q)=\frac{1}{n}\sum_{ab=n} b f^{[a]}|U_b(q^a)
\end{equation}
using the extended definition of the replicate. It may sound a bit vague, especially his description of a new algebra $\A^{[a]}$.  But if we take $\A=M_0^{!,\infty}(\G)$ and $\A^{[p]}=M_0^{!,\infty}(\G^{(p)})$ for a prime $p$, then the resulting replicates are same as ours: Suppose $f\in M_0^{!,\infty}(\G)$ and $f^{(p)}\in M_0^{!,\infty}(\G^{(p)})$.
If $p\nmid N$, then both $\frac1p T(p)$ and $\hat{T}_p$ are the classical Hecke operators, and  hence $f=f^{[p]}=f^{(p)}.$
If $p\mid N$, then by \eqref{deft}, \eqref{heckedef} and \eqref{def-fn}, we have $f|T(p)(\t)=f|U^*_p(\t)+f^{(p)}(p\t)$, while $pf|T_p(\t)=pf|U_p(\t)+f(p\t)$.  Since $f$ and $f^{(p)}$ have the same principal part by \eqref{def-fn}, $P_{\mathbb A^{[p]}}(f)=f^{(p)}$.
Accordingly, $f^{[p]}=f^{(p)}$ in either case.

\begin{proof}[Proof of Theorem \ref{ffp}] Theorem \ref{ffp} holds by Theorem \ref{fp-con} and the arguments following Theorem \ref{fp-con}.  
\end{proof}

\section{Congruences for coefficients of $f_{\G,m}(\t)$}

Throughout this section, we assume $X(\G)$ is of genus $g\geq 1$ and $i\infty$ is not a Weierstrass point of $X(\G)$.
Using our Hecke operator, we obtain several congruence relations that the coefficients of $f_{\G,m}$ satisfy. The first one is found by adopting the arguments in \cite{JT}.
\begin{prop}\label{congp}  Let $p\nmid N$ be a prime and $p>g$. Then for any positive integers $r$, $m$ and $n$ with $m>g$ and $p\nmid n$, we have
$$a_{\G}(mp^r,n)+\sum_{l=1}^g a_\G(mp^{r-1},-l)a_\G(lp,n)\equiv 0\pmod p,$$
where these coefficients are all integers.\end{prop}

\begin{proof}
For a prime $p\nmid N$, it follows from \eqref{deft} and \eqref{heckedef} that
$$f_{\G,m}|T(p)(\t)=f_{\G,m}|U^*_p(\t)+f_{\G,m}(p\t).$$ 
Comparing this with \eqref{fpm-hecke1}, we obtain that
\begin{equation}\label{cong0}
pa_{\G}(m,np)+a_{\G}(m,\frac{n}{p})=a_{\G}(mp,n)+\sum_{l=1}^g a_\G(m,-l)a_\G(lp,n)+pa_{\G}(\frac{m}{p},n). 
\end{equation}
Thus we have the following congruence relation of the Fourier coefficients of $f_{\G,m}$'s:
 \begin{equation}\label{cong1}
a_{\G}(m,np)-a_{\G}(\frac{m}{p},n)=p^{-1}\left(a_{\G}(m{p},n)+\sum_{l=1}^g a_\G(m,-l)a_\G(lp,n)-a_{\G}(m,\frac{n}{p})\right).
\end{equation}
For any $r\in \mathbb{N}$ and $1\leq i\leq r-1$, we replace $m$ with $mp^i$ and $n$ with $np^{r-i-1}$ in (\ref{cong1}) so that we have
 \begin{eqnarray}\label{cong2}
&&p^{-i}\left(a_{\G}(mp^i,np^{r-i})-a_{\G}(mp^{i-1},np^{r-i-1})\right)\cr
&&\quad =p^{-i-1}\left(a_{\G}(mp^{i+1},np^{r-i-1})+\sum_{l=1}^g a_\G(mp^i,-l)a_\G(lp,np^{r-i-1})-a_{\G}(mp^i,{n}{p}^{r-i-2})\right).
\end{eqnarray}
Next, we replace $n$ with $np^{r-1}$ in (\ref{cong1}) to obtain
 \begin{eqnarray}\label{cong3}
&&a_{\G}(m,np^{r})-a_{\G}(\frac{m}{p},np^{r-1})\cr
&&\quad =p^{-1}\left(a_{\G}(m{p},np^{r-1})+\sum_{l=1}^g a_\G(m,-l)a_\G(lp,np^{r-1})-a_{\G}(m,np^{r-2})\right).
\end{eqnarray}
Using (\ref{cong2}) $(r-1)$ times on the right-hand side of (\ref{cong3}), we have

\begin{eqnarray}\label{cong4}
a_{\G}(m,np^r)&-&a_{\G}(\frac{m}{p},np^{r-1})=p^{-r}\left(a_{\G}(mp^r,n)-a_{\G}(mp^{r-1},\frac{n}{p})\right)\cr
&&\ +p^{-1}\sum_{l=1}^g a_\G(m,-l)a_\G(lp,np^{r-1})+p^{-2}\sum_{l=1}^g a_\G(mp,-l)a_\G(lp,np^{r-2})\cr
&&\ +p^{-3}\sum_{l=1}^g a_\G(mp^2,-l)a_\G(lp,np^{r-3})+\cdots+p^{-r}\sum_{l=1}^g a_\G(mp^{r-1},-l)a_\G(lp,n).
\end{eqnarray}
Multiplying both sides by $p^r$ proves the theorem.
\end{proof}

The generalized Hecke operator allows one to derive a stronger result than Proposition \ref{congp}.

\begin{prop}\label{congp2} Let $p\nmid N$ be a prime. Then for any positive integers $r$, $n$ and $m>g$ with $p\nmid m$ and $p\nmid n$, we have
$$a_{\G}(mp^r,n)+\sum_{l=1}^g \sum_{p^j\mid(l,p)} p^j a_{\G}(m,-l)a_{\G}(lp^{r-2j},n)\equiv 0\pmod{p^r},$$
where these coefficients are all integers.
\end{prop}

\begin{proof}
By \eqref{deft} and \eqref{heckedef}, we have
\begin{equation}\label{1c}f_{\G,m}|T(p^r)(\t)=\sum_{i=0}^{r} f_{\G,m}|{U^*_{p^{r-i}}}(p^i\t).\end{equation}
Hence the coefficient of $q^n$ in $f_{\G,m}|T(p^r)(\t)$ must be 
\begin{equation}\label{eq1}
\sum_{i=0}^{r}p^{r-i}a_\G(m,p^{r-i}\frac{n}{p^i}),
\end{equation}
where $a_\G(m,p^{r-i}\frac{n}{p^i})=0$ if $p^i\nmid n$.
Comparing \eqref{eq1} with the right hand side of  (\ref{fpm-hecke1}), we have 
\begin{equation}\label{eq3}
\sum_{i=0}^{r}p^{r-i}a_\G(m,p^{r-i}\frac{n}{p^i})=\sum_{p^j|(m,p^r)}p^{j}a_\G(\frac{p^{r} m}{p^{2j}},n)+\sum_{l=1}^g \sum_{p^j\mid(l,p)} p^j a_\G(m,-l)a_\G(\frac{lp^r}{p^{2j}},n).
\end{equation}
Since $p\nmid n$ and $p\nmid m$, \eqref{eq3} is simplified to
$$p^ra_\G(m,p^r n)=a_\G(p^r m,n)+\sum_{l=1}^g \sum_{p^j\mid(l,p)} p^j a_\G(m,-l)a_\G(\frac{lp^r}{p^{2j}},n),$$
and the result follows.
\end{proof}

Note that  the right hand side of \eqref{eq3} is congruent modulo $p$ to
$$a_\G(p^r m,n)+\sum_{l=1}^g a_\G(m,-l)a_\G(lp^r,n).$$
Thus the following corollary holds for all $m$, including when $p\mid m$. 

\begin{cor}\label{congm} 
Let $p\nmid N$ be a prime. Then for any positive integers $r$, $n$ and $m>g$ with $p\nmid n$, we have
$$a_\G(p^r m,n)+\sum_{l=1}^g a_\G(m,-l)a_\G(lp^r,n)\equiv 0\pmod{p},$$
where these coefficients are all integers.
\end{cor}

By Proposition \ref{congp} and Corollary \ref{congm}, we have another congruence relation of coefficients of $f_{\G,m}$ modulo $p$.

\begin{cor}\label{m}
Let $p\nmid N$ be a prime and $p>g$. Then for any positive integers $r$, $n$ and $m>g$ with $p\nmid n$, we have
$$\sum_{l=1}^g a_\G(mp^{r-1},-l)a_\G(lp,n)\equiv \sum_{l=1}^g a_\G(m,-l)a_\G(lp^r,n)\pmod{p},$$
where these coefficients are all integers.
\end{cor}
\vspace{.1cm}

\begin{proof}[Proof of Theorem \ref{Tcong}]
If we give a restriction $p>g$ to Proposition \ref{congp2}, it proves Theorem \ref{Tcong} when $p\nmid N$.
It thus remains to prove that the congruence holds when $p\mid N$. But when $p\mid N$, $r,m\in\mathbb N$ with $p, m>g$ and $p\nmid m$, we already have in \eqref{fpm-hecke2} that
\begin{equation}\label{eq:hecke1}
f_{\G,m}|T(p^r)(\t)= f_{\G,p^rm}(\t)+\sum_{l=1}^{g} a_{\G}(m,-l)f_{\G,lp^r}(\t).
\end{equation}

On the other hand, by \eqref{ftpk}, we have
\begin{equation}\label{eq:hecke2}
f_{\G,m}|T(p^r)(\t)=f_{\G,m}|U^*_{p^r}(\t)+\sum_{i=1}^{r} f_{\G,m}^{(p^i)}|{U^*_{p^{r-i}}}(p^i\t).
 \end{equation}
When $p\nmid n$, the coefficient of $q^n$ of the first summand  in \eqref{eq:hecke2} is congruent to 0 modulo $p^r$, while the second summand does not contribute to the coefficient of $q^n$, because $i>0$. By comparing the coefficients of $q^n$ in \eqref{eq:hecke1} and \eqref{eq:hecke2}, we obtain \eqref{eq:cong} and complete the proof of Theorem \ref{Tcong}.
\end{proof}
Recall that \eqref{eq:hecke1} was derived from \eqref{a2}, in which the first equation is stated as
\begin{equation}\label{a22}f_{\G,m}|T(p^r)(\t)=\sum_{p^j|(m,p^r)} p^j J_{\G,\frac{p^rm}{p^{2j}}}^{(p^j)}(\t)+\sum_{l=1}^{g} a_{\G}(m,-l)J_{\G,lp^r}(\t).\end{equation}
In modulo $p$,  the right hand sides of \eqref{a22} and \eqref{eq:hecke1} are equal.  Hence in case $p\mid N$, whether $p\mid m$ or not, we obtain from \eqref{eq:hecke2} that
\begin{equation}\label{cong00}
f_{\G,p^rm}(\t)+\sum_{l=1}^{g} a_{\G}(m,-l)f_{\G,lp^r}(\t)\equiv f_{\G,m}^{(p^r)}(p^r\t) \pmod p.
\end{equation}

Thus the congruence in Corollary \ref{congm} when $p>g$  is generalized to include the case $p\mid N$ as follows: 

\begin{cor}\label{11cong} Let $p$ be any prime with $p>g$. Then for any positive integers $r$, $n$ and $m>g$ with $p\nmid n$, we have
$$a_{\G}(p^rm,n)+\sum_{l=1}^g a_\G(m,-l)a_\G(lp^r,n)\equiv 0\pmod{p},$$
where these coefficients are all integers.
\end{cor}

\section{Generating function of $J_{\G,n}(\t)$}

Let $q_1=e^{2\pi i \t_1}$ and $q_2=e^{2\pi i \t_2}$ and write 
 $f_{\G,m}(q_1):=f_{\G,m}(\t_1)$ and $h_{\G,n}(q_2):=h_{\G,n}(\t_2)$. 
Define 
\begin{equation}F_\G(q_1,q_2):=\sum_{m=0}^\infty f_{\G,m}(q_1)q_2^m=1+\sum_{m=g+1}^\infty f_{\G,m}(q_1)q_2^m \end{equation}
and
\begin{equation}H_\G(q_1,q_2):=\sum_{n=-g}^\infty h_{\G,n}(q_2)q_1^n. \end{equation}

It follows from Theorem \ref{Tgrid} that $F_\G(q_1,q_2)=-H_\G(q_1,q_2)$. This function $F_\G(q_1,q_2)$, the generating function of $f_{\G,m}(q_1)$, is a weight 2 meromorphic modular form with respect to $\t_2$ which has a simple pole at $\t_2=\t_1$ in the upper half-plane.

\begin{prop}\label{fgt}  Suppose the modular curve $X(\G)$ has genus $g$ and $i\infty$ is not a Weierstrass point on $X(\G)$.  If $q_1=e^{2\pi i \t_1}$ and $q_2=e^{2\pi i \t_2}$ for $\t_1,\tau_2\in\mathbb H$, then
\begin{equation*}\label{f3e}F_\G(q_1,q_2)=\frac{\sum_{\ell=-g}^{g+1} a_\G(g+1,-\ell)h_{\G,\ell}(q_2)+\sum_{\ell=1}^g f_{\ell+g+1}(q_1)h_{\G,-\ell}(q_2)+\sum_{j=g+1}^{2g}\sum_{\ell=j-g}^{g}a_\G(g+1,\ell-j)f_j(q_1)h_{-\ell}(q_2)}{f_{\G,g+1}(q_2)-f_{\G,g+1}(q_1)}.\end{equation*}
\end{prop}
\begin{proof}
First, we compute $f_{\G,{g+1}}(q)f_{\G,m}(q)$ for each $m\geq g+1$.
\begin{align*}
f_{\G,g+1}(q)&f_{\G,m}(q)=\lt(q^{-g-1}+\sum_{k=-g}^\infty a_{\G}(g+1,k)q^k\rt)\lt(q^{-m}+\sum_{\ell=-g}^\infty a_{\G}(m,\ell)q^\ell\rt)\cr
=&q^{-(m+g+1)}+\sum_{k=-g}^\infty a_{\G}(g+1,k)q^{-m+k}+\sum_{\ell=-g}^\infty a_{\G}(m,\ell)q^{\ell-g-1}+\sum_{k,\ell=-g}^\infty a_{\G}(g+1,k)a_{\G}(m,\ell)q^{k+\ell}\cr
=&\lt(f_{\G,m+g+1}(q)-\sum_{i=-g}^\infty a_{\G}(m+g+1,i)q^{i}\rt)\cr
&+\lt(\sum_{j=g+1}^{m+g} a_{\G}(g+1,m-j)\lt(f_{\G,j}(q)-\sum_{i=-g}^\infty a_{\G}(j,i)q^{i}\rt)+\sum_{k\geq {m-g}} a_{\G}(g+1,k)q^{-m+k}\rt)\cr
& +\lt(\sum_{\ell=1}^g a_{\G}(m,-\ell)\lt(f_{\G,\ell+g+1}(q)-\sum_{i=-g}^\infty a_{\G}(\ell+g+1,i)q^{i}\rt)+\sum_{\ell=1}^\infty a_{\G}(m,\ell)q^{\ell-g-1}\rt)\cr
& +\lt(\sum_{j=g+1}^{2g}\sum_{\ell=j-g}^g a_{\G}(g+1,\ell-j)a_{\G}(m,-\ell)\lt(f_{\G,j}(q)-\sum_{i=-g}^\infty a_{\G}(j,i)q^{i}\rt)+\sum_{k+\ell\geq -g}a_{\G}(g+1,k)a_{\G}(m,\ell)q^{k+\ell}\rt)\cr
=&f_{\G,m+g+1}(q)+\sum_{j=g+1}^{m+g}a_{\G}(g+1,m-j)f_{\G,j}(q)+\sum_{\ell=1}^g a_{\G}(m,-\ell)f_{\G,\ell+g+1}(q)\cr
&\qquad+\sum_{j=g+1}^{2g}\sum_{\ell=j-g}^g a_{\G}(g+1,\ell-j)a_{\G}(m,-\ell)f_{\G,j}(q)\cr
&\qquad+\sum_{i=1}^g A_i q^{-i}+a_{\G}(g+1,m)+a_{\G}(m,g+1)+\sum_{i=-g}^g a_{\G}(g+1,-i)a_{\G}(m,i)+O(q),\cr
\end{align*}
where  $A_i\ (1\leq i\leq g)$ in the last equation are constants. 
Since there is no modular form in $M_0^{!,\i}(\G)$ with a pole at $i\i$ of the order less than $g+1$,  we conclude  that
\begin{align*}
&f_{\G,g+1}(q)f_{\G,m}(q)=f_{\G,m+g+1}(q)+\sum_{j=g+1}^{m+g}a_{\G}(g+1,m-j)f_{\G,j}(q)+\sum_{\ell=1}^g a_{\G}(m,-\ell)f_{\G,\ell+g+1}(q)\cr
&\qquad+\sum_{j=g+1}^{2g}\sum_{\ell=j-g}^g a_{\G}(g+1,\ell-j)a_{\G}(m,-\ell)f_{\G,j}(q)+a_{\G}(g+1,m)+a_{\G}(m,g+1)+\sum_{i=-g}^g a_{\G}(g+1,-i)a_{\G}(m,i).\cr.
\end{align*}
Multiplying both sides by $q_2^m$ and summing on $m$ from $m=g+1$ to $m=\infty$, we have
\begin{align*}f_{\G,g+1}(q_1)&\sum_{m=g+1}^\infty f_{\G,m}(q_1)q_2^m=\sum_{m=g+1}^\i f_{\G,m+g+1}(q_1)q_2^m+\sum_{m=g+1}^\i \lt(\sum_{j=g+1}^{m+g}a_{\G}(g+1,m-j)f_{\G,j}(q_1)\rt)q_2^m\cr
&+\sum_{\ell=1}^g f_{\G,\ell+g+1}(q_1)\sum_{m=g+1}^{\i}a_{\G}(m,-\ell)q_2^m+\sum_{j=g+1}^{2g}\sum_{\ell=j-g}^g a_{\G}(g+1,\ell-j)f_{\G,j}(q_1)\sum_{m=g+1}^\i a_{\G}(m,-\ell)q_2^m\cr
&+\sum_{m=g+1}^{\i}\lt(a_{\G}(g+1,m)+a_{\G}(m,g+1)+\sum_{i=-g}^g a_{\G}(g+1,-i)a_{\G}(m,i)\rt)q_2^m.
\end{align*}
Let the first, second, third, fourth and fifth summand in the far right-hand side above as (A), (B), (C), (D) and (E), respectively.  Then 
using Theorem \ref{Tgrid} for (C), (D) and (E) below, we obtain
\begin{align*}
\mathrm{(A)}=&\frac{1}{q_2^{g+1}}\lt(F_\G(q_1,q_2)-1-\sum_{j=g+1}^{2g+1}f_{\G,j}(q_1)q_2^j\rt)=q_2^{-g-1}F_\G(q_1,q_2)-q_2^{-g-1}-f_{\G,g+1}(q_1)-\sum_{j=g+2}^{2g+1}f_{\G,j}(q_1)q_2^{j-g-1},\cr
&\cr
\mathrm{(B)}=&\sum_{j=g+1}^{2g}f_{\G,j}(q_1)q_2^j\sum_{m=g+1}^\i a_{\G}(g+1,m-j)q_2^{m-j}+\sum_{m=g+1}^\i\sum_{j=2g+1}^{m+g}a_{\G}(g+1,m-j)f_{\G,j}(q_1)q_2^m\cr
=&\sum_{j=g+1}^{2g}f_{\G,j}(q_1)q_2^j\lt(f_{\G,g+1}(q_2)-q_2^{-g-1}-\sum_{i=j-g}^{g}a_\G(g+1,-i)q_2^{-i}\rt)+\sum_{j=2g+1}^{\i}\sum_{k=-g}^\i  f_{\G,j}(q_1)q_2^j a_{\G}(g+1,k)q_2^k \cr
=&\sum_{j=g+1}^{2g}f_{\G,j}(q_1)f_{\G,g+1}(q_2)q_2^j-\sum_{j=g+1}^{2g}f_{\G,j}(q_1)q_2^{j-g-1}-\sum_{j=g+1}^{2g}f_{\G,j}(q_1)\sum_{i=j-g}^{g}a_\G(g+1,-i)q_2^{j-i}\cr
&+\lt(F_\G(q_1,q_2)-1-\sum_{j=g+1}^{2g}f_{\G,j}(q_1)q_2^{j}\rt)(f_{\G,g+1}(q_2)-q_2^{-g-1}),\cr
&\cr
\mathrm{(C)}=&-\sum_{\ell=1}^g f_{\G,\ell+g+1}(q_1)\sum_{m=g+1}^{\i}b_{\G}(-\ell,m)q_2^m=-\sum_{\ell=1}^g f_{\G,\ell+g+1}(q_1)(h_{\G,-\ell}(q_2)-q_2^\ell)\cr
=&-\sum_{\ell=1}^g f_{\G,\ell+g+1}(q_1)h_{\G,-\ell}(q_2)+\sum_{\ell=1}^g f_{\G,\ell+g+1}(q_1)q_2^\ell,\cr
&\cr
\mathrm{(D)}=&-\sum_{j=g+1}^{2g}\sum_{\ell=j-g}^g a_{\G}(g+1,\ell-j)f_{\G,j}(q_1)\sum_{m=g+1}^\i b_{\G}(-\ell,m)q_2^m\cr
=&-\sum_{j=g+1}^{2g}\sum_{\ell=j-g}^g a_{\G}(g+1,\ell-j)f_{\G,j}(q_1)(h_{\G,-\ell}(q_2)-q_2^{\ell}),\cr
&\cr
\mathrm{(E)}=&\sum_{m=g+1}^{\i}a_{\G}(g+1,m)q_2^m-\sum_{m=g+1}^{\i}b_{\G}(g+1,m)q_2^m-\sum_{i=-g}^g a_{\G}(g+1,-i) \sum_{m=g+1}^{\i} b_{\G}(i,m)q_2^m\cr
=&(f_{\G,g+1}(q_2)-q_2^{-g-1}-\sum_{i=-g}^g a_{\G}(g+1,i)q_2^{i})-(h_{\G,g+1}(q_2)-q_2^{-g-1})-\sum_{i=-g}^g a_{\G}(g+1,-i)(h_{\G,i}(q_2)-q_2^{-i})\cr
=&f_{\G,g+1}(q_2)-h_{\G,g+1}(q_2)-\sum_{i=-g}^g a_{\G}(g+1,-i)h_{\G,i}(q_2).
\end{align*}
 After adding (A) through (E), we find that
\begin{align*}
f_{\G,g+1}(q_1)F_\G(q_1,q_2)=&f_{\G,g+1}(q_2)F_\G(q_1,q_2)-h_{\G,g+1}(q_2)-\sum_{\ell=-g}^g a_\G(g+1,-\ell)h_{\G,\ell}(q_2)\\
&\qquad-\sum_{\ell=1}^g f_{\ell+g+1}(q_1)h_{\G,-\ell}(q_2) -\sum_{j=g+1}^{2g}\sum_{\ell=j-g}^{g}a_\G(g+1,\ell-j)f_j(q_1)h_{-\ell}(q_2),
\end{align*}
which proves the Proposition, because $a_\G(g+1,-g-1)=1.$
\end{proof}

\begin{proof}[Proof of Theorem \ref{genf}]
Since $\frac{d}{d\t}f_{\G,g+1}(\t)\in S_2^{!,\i}(\G)$, it is a linear combination of $h_{\G,n}(\t)$ ($n\geq -g$). Hence 
 \begin{align}\label{fgder} \nonumber
 -\frac{1}{2\pi i}\frac{d}{d\t_2}(f_{\G,g+1}(\t_2))&=(g+1)q_2^{-g-1}+\sum_{\ell=1}^g \ell a_{\G}(g+1,-\ell)q_2^{-\ell}-\sum_{n=1}^\i n a_{\G}(g+1,n) q_2^n\\ 
 &=(g+1)h_{\G,g+1}+\sum_{\ell=1}^g \ell a_{\G}(g+1,-\ell)h_{\G,\ell}-\sum_{n=1}^g n a_{\G}(g+1,n)h_{\G,-n}\\  \nonumber
&= \sum_{\ell=-g}^{g+1} \ell a_{\G}(g+1,-\ell)h_{\G,\ell}.
 \end{align}
Employing  \eqref{fgder} and Proposition \ref{fgt}, we find another representation of $F_\G(q_1,q_2)$ in Theorem \ref{genf}.
\end{proof}

\begin{proof}[Proof of Theorem \ref{genj}]

Recall from \eqref{fjg} that for $m\geq g+1$ 
$$f_{\G,m}=J_{\G,m}+\sum_{l=1}^{g} a_{\G}(m,-l)J_{\G,l}$$
and also recall that $f_{\G,m}=0$ for $1\leq m\leq g$, $f_{\G,0}=1$ and $J_{\G,0}=1$.
Hence using Theorem \ref{Tgrid} in the third equality below, we obtain 
\begin{align*}
F_\G(p,q)=&\sum_{m=0}^\infty f_{\G,m}(p)q^m=1+\sum_{m=g+1}^\infty (J_{\G,m}(\t_1)+\sum_{l=1}^{g}a_{\G}(m,-l)J_{\G,l}(\t_1))q^m\\ \nonumber
=&1+\sum_{m=g+1}^\infty J_{\G,m}(\t_1)q^m-\sum_{m=g+1}^\infty\sum_{l=1}^{g}J_{\G,l}(\t_1) b_{N}(-l,m)q^m\\ \nonumber
=&1+\sum_{m=g+1}^\infty J_{\G,m}(\t_1)q^m-\sum_{l=1}^{g}J_{\G,l}(\t_1)(h_{\G,-l}(\t_2)-q^l)\\ \nonumber
=&\sum_{m=0}^\infty J_{\G,m}(\t_1)q^m-\sum_{l=1}^{g}J_{\G,l}(\t_1)h_{\G,-l}(\t_2). \nonumber
\end{align*}
This proves our last theorem.
\end{proof}

\end{document}